\newif\ifHAL
\HALtrue

\ifHAL
\documentclass[10pt,a4paper]{article}
\usepackage{a4wide,geometry}
\else
\documentclass[10pt]{siamltex}
\fi
\usepackage{amsmath}
\usepackage[utf8]{inputenc}
\usepackage[T1]{fontenc}
\usepackage[english]{babel}
\usepackage{color,graphicx}
\usepackage{amssymb}
\usepackage{amsbsy}
\usepackage{amssymb}      
\usepackage{amsmath}
\usepackage{array}
\usepackage{caption}
\usepackage{subcaption}
\usepackage{enumitem}
\usepackage{accents}

\ifHAL
\usepackage{enumitem}
\usepackage{amsthm}
\newtheorem{theorem}{Theorem}[section]
\newtheorem{assumption}[theorem]{Assumption}

\newtheorem{lemma}[theorem]{Lemma}
\newtheorem{corollary}[theorem]{Corollary}
\newtheorem{remark}[theorem]{Remark}
\newtheorem{definition}[theorem]{Definition}

\else
\newtheorem{remark}[theorem]{Remark}
\newtheorem{assumption}[theorem]{Assumption}
\usepackage{fancyhdr}
\usepackage{caption}
\usepackage{subcaption}

\fi

\usepackage{xcolor}
\usepackage{scalerel}

\input mydef.sty

\begin{document}

\ifHAL

\title{${\lowercase{hp}}$-a posteriori error estimates for hybrid high-order methods applied to biharmonic problems}

\author{Zhaonan Dong\footnotemark[1], \quad  Alexandre Ern\footnotemark[2], \quad Tanvi Wadhawan\footnotemark[1]}

\footnotetext[1]{Inria, 48 rue Barrault, 75647 Paris, France and CERMICS, CNRS, ENPC, Institut Polytechnique de Paris, 6 \& 8 avenue B.~Pascal, 77455 Marne-la-Vall\'{e}e, France}

\footnotetext[2]{CERMICS, CNRS, ENPC, Institut Polytechnique de Paris, 6 \& 8 avenue B.~Pascal, 77455 Marne-la-Vall\'{e}e, France and Inria, 48 rue Barrault, 75647 Paris, France}

\date{\today}

\fi

\maketitle

\begin{abstract}
We derive a residual-based $hp$–a posteriori error estimator for hybrid high-order (HHO) methods on simplicial meshes applied to the biharmonic problem posed on two- and three-dimensional polytopal Lipschitz domains. The a posteriori error estimator hinges on an error decomposition into conforming and nonconforming components. To bound the nonconforming error, we use a $C^1$-partition of unity constructed via Alfeld splittings, combined with local Helmholtz decompositions on vertex stars where the key contribution is to show that the stability constant only depends on the mesh shape-regularity.  For the conforming error, we design two residual-based estimators, each associated with a specific interpolation operator. In the first setting, the upper bound on the conforming error involves only the stabilization term and the data oscillation, but hinges on an assumption that we verify numerically.
In the second setting, the bound additionally incorporates bulk residuals, normal flux jumps, and tangential jumps. Numerical experiments confirm the theoretical findings on the error upper bound and also illustrate numerically that the proposed estimators lead to moderate effectivity indices.
\end{abstract}

\section{Introduction}\label{Intro}

Fourth-order problems are encountered in many engineering applications, including optimal control problems, micro-electro mechanical systems,  thin plate elasticity, and hyperviscous effects in fluid models.  A prime example of a fourth-order problem is the biharmonic problem with so-called clamped boundary conditions:
\begin{subequations} \label{def: Model problem}
\begin{align}
\Delta^2 u = f ~~&\text{in}~\Omega, \\
u=0~ ~~&\text{on} ~\partial\Omega, \\
\b{n} {\cdot} \nabla u=0~ ~~&\text{on} ~\partial\Omega,
\end{align}
\end{subequations}
where $\Omega \subset \mathbb{R}^d,~d\in\{2,3\}$, is a bounded polytopal connected domain with Lipschitz  boundary $\partial \Omega$, $\b{n} $ denotes the unit normal vector on  $\partial \Omega$, and $f\in L^2(\Omega)$.

The hybrid high-order method (HHO) is a  discretization technique that is applicable to a wide range of partial differential equations. It was initially developed for linear diffusion \cite{DiPEL:14} and elasticity  \cite{DiPEr:15}. HHO methods are formulated using broken polynomial spaces on the mesh cells and the mesh faces. The two essential components in designing HHO methods are a local reconstruction operator and a local stabilization operator, both defined locally on each mesh cell. HHO methods offer several appealing features, including support for polytopal meshes, optimal error estimates, local conservation properties, and enhanced computational efficiency due to compact stencils and the local elimination of cell unknowns via static condensation. HHO methods exhibit close connections to hybridizable discontinuous Galerkin (HDG) and weak Galerkin (WG) methods \cite{CoDPE:16}. These connections have been exploited to establish unified convergence analyses for the biharmonic problem and the acoustic wave equation \cite{DongErn2021biharmonic,ErnSteins2024}. Links between HHO and the nonconforming virtual element method (ncVEM) are discussed in \cite{CoDPE:16,Bridging_HHO,cicuttin2021hybrid}.

A posteriori error estimates and adaptive strategies for standard Galerkin methods applied to fourth-order problems have been the subject of extensive research over the past two decades. Early contributions include conforming finite element approximations~\cite{neittaanmaki2001posteriori} and Morley plate elements~\cite{da2007posteriori}. $C^0$-interior penalty discontinuous Galerkin (IPDG) methods were analyzed in~\cite{BrennerGudiSung2010} in the quadratic case and in~\cite{GeorgoulisHoustonVirtanen2011} for general order in two space dimensions. Other notable contributions include continuous and discontinuous Galerkin methods for the Kirchhoff--Love plate~\cite{HansboLarson2011} and the Ciarlet--Raviart formulation of the biharmonic problem~\cite{CharbonneauDossouPierre1997}. A posteriori error estimates for several lowest-order nonconforming finite element methods applied to biharmonic problems in two and three dimensions are established in~\cite{CarGraNat2024}. HHO methods were addressed in \cite{LiangTran2025}, where a posteriori error estimates were derived for a special HHO method employing additional one-dimensional edge unknowns for biharmonic problems in three space dimensions. Therein, the tool to control the nonconforming error is an averaging operator which maps discontinuous or continuous piecewise polynomial functions to a $C^1$-piecewise polynomial space. This allows one to bound the nonconforming error by the jumps of the underlying functions in an $h$-optimal way, but $p$-optimality is lost.

The first $hp$–a posteriori error estimate for Galerkin methods applied to biharmonic PDEs was established in~\cite{DongMascottoSutton2021}. In that work, an $hp$–a posteriori error estimate is derived using a novel approach in which the nonconforming error is controlled via a global (non-polynomial) $H^2$-potential reconstruction function, rather than a nodal averaging operator, on two- and three-dimensional, simply connected, Lipschitz domains. The resulting a posteriori error bound is $h$-optimal and $p$-suboptimal by $\frac32$-order, which matches the known $p$-suboptimality for the IPDG method. We observe that invoking a global Helmholtz decomposition leads to a constant in the error bound that deteriorates with the number of holes in the domain $\Omega$, as observed in another context in ~\cite{bertrand2023stabilization}. Another recent work \cite{chaumontfrelet:hal-05176686} studied the biharmonic problem on a simply connected planar domain and analyzed a symmetric IPDG method. By exploiting the “div–div” complex, the authors defined a polynomial lifting operator mapping the nonconforming error to a high-order ${\rm \b{H}(divdiv)}$-conforming finite element space, resulting in an error estimator that does not contain the stabilization term.

In the present work, we derive residual $hp$–{a posteriori} error estimates for the HHO methods introduced in~\cite{DongErn2021biharmonic} for the biharmonic problem posed on two- and three-dimensional Lipschitz domains with general topology. To the best of our knowledge, this is the first work to employ a $C^1$-partition of unity, combined with a local Helmholtz decomposition on each vertex star to bound the nonconforming error in the $hp$-setting. A similar idea to bound the nonconforming error in the $h$-setting can be found in \cite{GallistlTian24} with a different local Helmholtz decomposition regarding local boundary conditions. The $C^1$-partition of unity can be constructed using existing $C^1$-conforming finite element or composite finite element spaces, such as the Argyris or HCT elements. In the present work, we employ the $C^1$-partition of unity introduced in~\cite{Walk14}. The role of the partition of unity is to localize the Helmholtz decomposition, thereby avoiding global stability constant that depends on the number of holes present in the domain $\Omega$.
Furthermore, although Helmholtz decompositions on simply connected domains are already available in the literature, as in \cite{da2007posteriori} and \cite{PaulyZulehner2020} for the two- and three-dimensional cases, respectively, a crucial novelty of the present work is to establish that the stability constant of such decompositions when applied on vertex patches only depends on the mesh shape-regularity.
Finally, we believe that the present approach to bound the nonconforming error is of broader interest beyond HHO methods, as it is potentially applicable to other nonconforming finite element methods, such as IPDG methods.

The second main contribution of the paper is to obtain two upper bounds on the conforming error. Each upper bound is derived using a specific interpolation operator. The first approach uses the interpolation operator associated with the canonical hybrid finite element (see, e.g., \cite[Section 7.6]{Ern_Guermond_FEs_I_2021}) whose main advantage is that its combination with the HHO reconstruction operator leads to the $H^2$-elliptic projection. The resulting upper bound on the conforming error is $p$-suboptimal by at most one order, but is actually expected to be only $\frac12$-order suboptimal under a reasonable assumption on the interpolation error which we verify numerically in two and three dimensions. The remarkable fact about the resulting upper bound is that it involves only stabilization and data oscillation terms. The alternative (and somewhat more classical) approach to bound the conforming error hinges on the Babu\v{s}ka--Suri interpolation operator. The resulting upper bound is $p$-suboptimal by $\frac12$-order without any assumption, but at the price of introducing additional error indicator terms, namely bulk residual and jump contributions. We emphasize that both approaches rely on interpolation operators constructed on simplicial meshes.

The remainder of the paper is organized as follows. In Section~\ref{sec:Analysis Tools}, we present the weak formulation of the model problem and recall the main $hp$-approximation tools employed in the analysis. In Section~\ref{sec:HHO}, we introduce the HHO method and the two interpolation operators mentioned above. In Section~\ref{sec: Local Helmholtz Decomposition}, we derive a novel stability estimate on local Helmholtz decompositions on vertex stars. The main result is Lemma~\ref{lemma:Helmholtz-dec}. Section~\ref{sec:apost} is devoted to the residual-based $hp$-{a posteriori} error upper bound. The main result is Theorem~\ref{thm:upperbound}. Numerical experiments illustrating the theoretical findings are presented in Section~\ref{sec:Numerical example}.  Finally, Section~\ref{sec:proofs} gathers some technical proofs supporting the $hp$-{a posteriori} error analysis. We mention that we do not state localized lower error bounds. Indeed, such bounds follow from straightforward adaptations of the arguments invoked in~\cite{DongMascottoSutton2021} for dG methods. Moreover, these bounds exhibit a pessimistic $p$-suboptimality (compared to numerical observations) as their proof involves $C^1$-bubble functions and $H^2$-extension operators which cannot be constructed using the $hp$-techniques from~\cite{MeWo01}.

\section{Analysis tools}\label{sec:Analysis Tools}
In this section, we introduce essential notation at the continuous and discrete levels, formulate the weak problem, and recall some useful results from the literature.

\subsection{Basic notation}
We adopt standard notation for Lebesgue and Sobolev spaces. Let $S \subset \mathbb{R}^d$, $d \in \{2,3\}$, be an open, bounded, Lipschitz set. For scalar-, vector-, or tensor-valued fields, we denote the $L^2$-inner product as $(\bullet,\bullet)_S$. We employ boldface font to denote vector-valued fields, and we use an additional undertilde notation, e.g., $\ut{\b{A}}$, for tensor-valued fields. The (weak) gradient of a scalar-valued function $v$ is denoted as $\nabla v$, and its (weak) Hessian as $\hes v$. Let $\b{n}_S$ denote the unit outward normal vector on the boundary $\partial S$ of~$S$. For any smooth functions $v$ and $w$, we have the following integration by parts formula:
\begin{align}\label{intergration by parts}
(\lapp v,  w)_{S}  = (\hes v , \hes w)_{S} +(\b{n}_S {\cdot}\grad \Delta v, w)_{\partial{S}}   - (\hes v\b{n}_S, \grad{w})_{\partial{S}} .
\end{align}
Let $\utP_{\partial S}:=\utI-\b{n}_S {\otimes} \b{n}_S\in \RR^{d\times d}$, where $\utI$ denotes the identity matrix in $\RR^{d\times d}$. Whenever the context is unambiguous (e.g., when a term is evaluated within the inner product or the norm in $L^2(\partial S)$), we denote $ \partial_n v : = \b{n}_S {\cdot} \nabla v$ the (scalar-valued) outward normal derivative of $v$ on $\partial S$, and $\partial_t v : = \utP_{\partial S} \nabla v$ its ($\RR^{d}$-valued) tangential derivative. Similarly, $\partial_{nn} v: =  \b{n}_S{\cdot}\hes v \b{n}_S$, $\partial_{nt}v : = \utP_{\partial S} \hes v \b{n}_S$ and $\partial_{tt} v:=  \utP_{\partial S}  \hes v \utP_{\partial S} $ correspond to the (scalar-valued) normal-normal, ($\RR^{d}$-valued) normal-tangential and ($\RR^{d\times d}$-valued) tangential-tangential components of the Hessian of $v$,  respectively.  With these conventions, the integration by parts formula~\eqref{intergration by parts} can be rewritten as
\begin{align}\label{eq:IPP_bis}
(\lapp v,  w)_{S} = (\hes v , \hes w)_{S}+ (\partial_n\lap v, w)_{\partial{S}} - (\partial_{nn}v, \partial_nw)_{\partial{S}} - (\partial_{nt}v, \partial_tw)_{\partial{S}} .
\end{align}

\subsection{Weak formulation}
The weak formulation of the model problem~\eqref{def: Model problem} reads as follows: Find $u \in H^2_0(\Omega)$ such that
\begin{align}\label{weakform}
(\hes u , \hes v)_{\Omega} = (f, v)_{\Omega}\qquad\forall v \in  H^2_0(\Omega).
\end{align}
The well-posedness of \eqref{weakform} is proven, e.g., in \cite[Section 1.5]{GirRa:86}.

\subsection{Mesh}
Let $\mesh$ denote a conforming simplicial mesh of the domain $\Omega$.
A generic mesh cell is denoted by $T \in \mesh$, its diameter by $h_T$, and its unit outward normal by $\n_T$. All the mesh cells are generated from a reference simplex by an affine geometric mapping. The mesh faces are collected in the set $\Fall$, which is decomposed as $\Fall = \Fint \cup \Fb$, where $\Fint$ is the collection of mesh interfaces (shared by two distinct mesh cells) and  $\Fb$ the collection of mesh boundary faces. We denote by $h_F$ the diameter of a generic mesh face $F\in\Fall$. Every $F\in \Fint$ is oriented by a unit normal vector $\n_F$ with arbitrary, but fixed, direction, whereas every $F\in\Fb$ is oriented by $\n|_F$. The boundary $\dK$ of every mesh cell $T\in\mesh$ is split as ${\dK}=\dKi\cup \dKb$ with obvious notation, and the mesh faces composing $\dK$ are collected in the set $\FK$, which is split as $\FK=\FK^{\rm i} \cup  \FK^{\rm b}$.
The mesh vertices are collected in the set $\vertice$, which is split as $\vertice = \verticei \cup \verticeb$. For all $\ba \in \vertice$, we denote by $\mesha$
the collection of mesh cells that share the vertex $\ba$,
and by $\oma$ the corresponding open subdomain, also referred to as vertex patch.

Let $\ell \ge 0$ be an integer number. For all $T\in\mesh$, we denote by $\mathbb{P}^{\ell}(T)$ the space of $d$-variate polynomials of degree at most $\ell$ restricted to $T$, by $\Pi^\ell_T$ the $L^2$-orthogonal projection onto $\mathbb{P}^{\ell}(T)$ and by $\bbP^{\ell}(\mesh) := \{ v_h \in L^2(\Omega) \;|\; v_h|_T \in \bbP^{\ell}(T) \}$ the broken polynomial space of order $\ell$ on the mesh. We adopt a similar notation for all $F\in\Fall$, leading to the polynomial space $\mathbb{P}^\ell(F)$, the $L^2$-orthogonal projection $\Pi^\ell_F$, and the broken polynomial space $\bbP^{\ell}(\Fall)$. At some occasions, we also consider $\ell\le -1$, in which case $\mathbb{P}^{\ell}(T):=\{0\}$ and $\Pi^{\ell}_{T}$ identically maps to the zero function.

Let $s\ge0$ be a real number. We define the broken Sobolev space $H^s(\mesh)
:= \bigl\{\, w \in L^2(\Omega) \;|\; w|_T \in H^s(T) \; \forall\, T \in \mesh \,\bigr\}$. For all $w \in H^s(\mesh)$, $s>\frac12$, its jump across any mesh interface
$F = \dK_1 \cap \dK_2 \in \Fint$ is defined as
$\sjump{w} := w|_{T_1}|_F - w|_{T_2}|_F$,
where $\n_F$ is oriented from $T_1$ to $T_2$. On every boundary face
$F = \dK \cap \partial \Omega \in \Fb$, we set $\sjump{w} := w_T|_F$.
Finally, the broken gradient $\nabla_{\mesh}$ and broken Hessian $\hes_{\mesh}$ are
defined as the gradient and
Hessian operators acting cellwise on $H^1(\mesh)$ and $H^2(\mesh)$, respectively.

\subsection{Analysis tools}
In this section, we recall several inequalities that are used in our $hp$-error analysis.
We use the notation $A \lesssim B$ to mean that $A \le C B$ for positive numbers $A$ and $B$, where $C$ denotes a generic (nondimensional) positive constant, whose value may vary at each occurrence provided it is independent of the mesh size $h$, the underlying polynomial degree, and the topology of the domain $\Omega$. The value of $C$ can depend on the shape-regularity parameter of the mesh and the space dimension.

\begin{lemma}[$hp$-discrete trace inequality]\label{lemma: Discrete trace  inequality}
For all $T \in \mesh$, all $v \in \bbP^p(T)$ with $p \ge 0$, and all $F\in\FK$, the following holds:
\begin{subequations}
\begin{align}\label{eq: discrete trace sharp_estimate}
\|v - \Pi^{n}_{T}(v)\|_{F}
\lesssim  \bigg\{\frac{(p-n)(p+1 + n + d)}{h_T}\bigg\}^{\frac12}
\|v - \Pi^{n}_{T}(v)\|_{T} \qquad \forall n\in\{-1,\ldots,p\}.
\end{align}
In particular, we have
\begin{align}\label{discrete}
\|v\|_{\dK} \lesssim \bigg\{\dfrac{(p+1)^2}{h_T}\bigg\}^{\frac{1}{2}}\|v\|_{T}.
\end{align}
\end{subequations}
\end{lemma}
\begin{proof}
The proof of~\eqref{eq: discrete trace sharp_estimate} can be found in \cite{DongWadhawan2026}, whereas the proof for the case $n=-1$ can be found in \cite{warburton2003constants}. \eqref{discrete} readily follows from~\eqref{eq: discrete trace sharp_estimate} by taking $n=-1$, summing over $F\in\FK$, and using that $p+d\lesssim p+1$.
\end{proof}

\begin{lemma}[$hp$-inverse inequality]\label{lemma: Discrete Inverse  inequality}
For all $T\in\mesh$ and all $v \in \mathbb{P}^p(T)$ with $p \geq 0$, the following holds:
\begin{align}\label{eq:discrete inverse}
\|\nabla v\|_{T} \lesssim \frac{p^2}{h_T} \|v\|_{T}.
\end{align}
\end{lemma}
\begin{proof}
A proof can be found in \cite[Theorem~4.76]{schwab}.
\end{proof}


\begin{lemma}[Global Babu\v{s}ka--Suri $hp$-interpolation operator]\label{lemma: hp-KM orginal}
Fix $\epsilon>0$.
For all $p\geq1$, there exists an interpolation operator $\BS^p: H^1_0(\Omega) \cap H^{\frac{d}{2} +\epsilon}(\Omega)  \rightarrow  \mathbb{P}^{p}(\mesh) \cap H^1_{0} (\Omega)$ such that, for all $r \in \{ 2, \ldots p\}$, all $m \in \{ 0, \ldots, r\}$, all $v\in H^1_0(\Omega) \cap H^{\frac{d}{2} +\epsilon}(\Omega) \cap H^r(\mesh)$, and all $T\in \mesh$, the following holds:
\begin{equation}\label{BS interpolation}
|v-\BS^p(v)|_{H^m(T)} \lesssim \bigg\{\dfrac{h_T}{p} \bigg\}^{r-m} \ell_\Omega^{-r} \|v\|_{H^r(T)},
\end{equation}
where $\|v\|_{H^r(T)}^2:= \sum_{n\in\{0,\ldots,r\}} \ell_\Omega^{2n} |v|_{H^n(T)}^2$ and $\ell_\Omega$ is a global length scale associated with $\Omega$ (e.g., its diameter) which is introduced for dimensional consistency.
\end{lemma}
\begin{proof}
A proof can be found in \cite{babuvska1987optimal} (without considering the global length scale $\ell_\Omega$).
\end{proof}

\begin{corollary}[Modified Babu\v{s}ka--Suri $hp$-interpolation operator]\label{cor: global interpolation of BS}
Fix $\epsilon>0$.
For all $p\geq2$, there exists an interpolation operator
$\mBS^p: H^1_0(\Omega) \cap H^{\frac32 +\epsilon}(\Omega) \rightarrow \mathbb{P}^{p}(\mesh) \cap H^1_{0} (\Omega)$ such that, for all $v\in H^1_0(\Omega) \cap H^{\frac32 +\epsilon}(\Omega) \cap H^{2}(\mesh)$ and all $T\in \mesh$, the following holds:
\begin{multline} \label{eq: Modified BS approximation}
\bigg\{\frac{p}{h_T}\bigg\}^{2}\|v - \mBS^p (v)\|_{T}
+ \bigg\{\frac{p}{h_T}\bigg\}^{\frac32}\|v - \mBS^p  (v)\|_{\dK}
+ \bigg\{\frac{p}{h_T}\bigg\} \|\nabla(v - \mBS^p  (v))\|_T \\
+ \bigg\{\frac{p}{h_T}\bigg\}^{\frac12}\|\partial_{n} (v - \mBS^p  (v))\|_{\dK}
+ \|\hes \mBS^p  (v)\|_{T} \lesssim \| \hes v\|_{T}.
\end{multline}
\end{corollary}
\begin{proof}
The idea is to set, for all $v\in H^1_{0}(\Omega)$,
\begin{align*}
\mBS^p(v) : = \mathcal{J}^2_{0h} (v) + \BS^p(v -  \mathcal{J}^2_{0h} (v) ),
\end{align*}
with $\mathcal{J}^2_{0h}:H^1_0(\Omega)\rightarrow \mathbb{P}^2(\mesh)\cap H^1_0(\Omega)$ denoting the (piecewise quadratic) $L^2$-stable quasi-interpolation operator with prescribed boundary conditions derived in \cite{ErnGuermond:17}. This operator satisfies, for all $v\in H^1_0(\Omega)\cap H^2(\mesh)$ and all $T\in\mesh$,
\begin{multline*}
h_T^{-2}\|v -\mathcal{J}^2_{0h} (v)\|_{T}
+
h_T^{-\frac32}\|v -\mathcal{J}^2_{0h} (v)\|_{\dK}
+
h_T^{-1}\|\nabla(v -\mathcal{J}^2_{0h} (v))\|_T \\
+
h_T^{-\frac12}
\|\partial_n(v -\mathcal{J}^2_{0h} (v))\|_{\dK}
+  \|\hes \mathcal{J}^2_{0h} (v)\|_{T}
\lesssim \|\hes v\|_{T}.
\end{multline*}
Invoking \eqref{BS interpolation}, we infer that
\begin{equation*}
\begin{split}
\bigg\{\frac{p}{h_T}\bigg\}^{2}\|v &- \mBS^p (v)\|_{T}  = \bigg\{\frac{p}{h_T}\bigg\}^{2} \|(v -  \mathcal{J}^2_{0h} (v)) - \BS^p(v -  \mathcal{J}^2_{0h}  (v) )\|_{T} \\
& \lesssim  \ell_\Omega^{-2} \|v -  \mathcal{J}^2_{0h} (v)\|_{H^2(T)} \\
& \lesssim \ell_\Omega^{-2} \|v -  \mathcal{J}^2_{0h} (v)\|_{T} +
\ell_\Omega^{-1} \|\nabla(v -  \mathcal{J}^2_{0h} (v))\|_{T} + \|\hes(v -  \mathcal{J}^2_{0h} (v))\|_{T}\\
& \lesssim \|\hes v\|_{T},
\end{split}
\end{equation*}
where the last bound follows from the above approximation properties of $\mathcal{J}^2_{0h}$ and $h_T\le \ell_\Omega$.
This establishes the estimate for the first term on the left-hand side of~\eqref{eq: Modified BS approximation}. The remaining terms are handled similarly.
\end{proof}

\section{Hybrid high-order discretization}\label{sec:HHO}

In this section, we present the main ideas underlying the HHO discretization and state some useful results for the forthcoming analysis.

\subsection{Local reconstruction and stabilization}

Let $k\ge0$ be the polynomial degree.
For every mesh cell $T \in \mesh,$ the local HHO space is
\begin{align}
\VkT := \mathbb{P}^{k+2}(T) \times \mathbb{P}^{k+2}(\FK) \times \mathbb{P}^{k}(\FK),
\end{align}
with $\mathbb{P}^{k}(\FK) := \times_{F\in\FK} \mathbb{P}^k(F)$.
A generic element $\widehat{v}_T:= \big( v_T, v_{\dK}, \gamma_{\dK} \big) \in \VkT$ is a triple, where the first component $v_T \in \mathbb{P}^{k+2}(T)$ aims at approximating the solution in $T$, the second component $v_{\dK} \in \mathbb{P}^{k+2}(\FK)$ its trace on $\dK$, and the third component $\gamma_{\dK} \in \mathbb{P}^{k}(\FK)$ the normal derivative (along $\n_T$) on $\dK$.

The first step in the devising of the HHO method is a local discrete reconstruction operator. For all $T \in \mesh$,  this operator $ R_T^{k+2}: \widehat{V}_T^k \rightarrow \mathbb{P}^{k+2}(T)$ is such that, for all $\widehat{v}_T := \big( v_T, v_{\dK}, \gamma_{\dK} \big)\in \VkT $, $R_T^{k+2}(\widehat{v}_T)\in \mathbb{P}^{k+2}(T)$ is uniquely determined by solving the following well-posed problem:
\begin{subequations}\label{reconstruct}
\begin{align}
(\hes  R_T^{k+2} (\widehat{v}_T),  \hes w)_{T} ={}& (\hes  v_T,  \hes w)_{T}  + (v_T-v_{\dK},\partial_n \Delta w)_{\dK}  \nonumber \\
&- ( \partial_n v_T - \gamma_{\dK},\partial_{nn} w)_{\dK}
-  (\partial_t(v_T- v_{\dK}),\partial_{nt} w)_{\dK}, \label{reconstruct_a}\\
( R_T^{k+2} (\widehat{v}_T), \xi)_T ={}& (v_T, \xi)_T, \label{reconstruct_b}
\end{align}
\end{subequations}
with test functions $w \in \mathbb{P}^{k+2}(T)$ in~\eqref{reconstruct_a} (one obtains
$0=0$ whenever $w\in\mathbb{P}^1(T)$)
and test functions $\xi\in \mathbb{P}^1(T)$ in~\eqref{reconstruct_b}.
Invoking~\eqref{eq:IPP_bis}, \eqref{reconstruct_a} can be rewritten as
\begin{equation}\label{Def: Reconstruction}
(\hes  R_T^{k+2} (\widehat{v}_T),  \hes w)_{T} = ( v_T,  \Delta^2 w)_{T} -  (v_{\dK},\partial_n \Delta w)_{\dK}
+( \gamma_{\dK},\partial_{nn} w)_{\dK} +   (\partial_t v_{\dK}, \partial_{nt} w)_{\dK}.
\end{equation}

The second devising step is a local stabilization bilinear form $S_{\dK}: \widehat{V}_T^k \times \widehat{V}_T^k \rightarrow \mathbb{R}$ such that,  for all $(\hat{v}_T, \hat{w}_T)\in \widehat{V}_T^k \times \widehat{V}_T^k$ with $\hat{v}_T:= (v_T,v_{\dK}, \gamma_{\dK})$ and
$\hat{w}_T:= (w_T,w_{\dK},\chi_{\dK})$,
\begin{subequations} \label{def:stab}
\begin{align}
S_{\dK} (\widehat{v}_T,  \widehat{w}_T) := {}& (k+2)^3\hbar_T^{-3} (v_{\dK} -v_T, w_{\dK} -w_T)_{\dK}\nonumber \\
&+(k+2)\hbar_T^{-1}(\gamma_{\dK} -\Pi^{k}_{\dK}(\partial_n v_T), \chi_{\dK} - \Pi^{k}_{\dK}(\partial_n w_T))_{\dK},
\end{align}
with the $hp$-scaling factor
\begin{equation}
\hbar_T := \frac{h_T}{k+2}.
\end{equation}
\end{subequations}

\begin{lemma}[Useful property]\label{prop}
For all $\widehat{v}_T \in \widehat{V}^k_T$ and all $T\in \mesh$, the following holds:
\begin{equation}\label{HHO relation}
\|\hes ( R_T^{k+2} (\widehat{v}_T)-v_T)\|^2_{T} \lesssim S_{\dK}( \widehat{v}_T, \widehat{v}_T).
\end{equation}
\end{lemma}
\begin{proof}
Using \eqref{reconstruct_a} and the Cauchy--Schwarz inequality, we deduce, for all $w \in \mathbb{P}^{k+2}(T)$, that
\begin{align*}
(\hes( R_T^{k+2} (\widehat{v}_T) -v_T), \hes w)_{T}  \leq {}&  \|\partial_n \Delta w\|_{\dK} \|v_T-v_{\dK}\|_{\dK}  +  \|\partial_{nt} w\|_{\dK} \|\partial_t(v_T- v_{\dK})\|_{\dK} \\
&+  \|\partial_{nn} w\|_{\dK} \|\Pi^k_{\dK}(\partial_n v_T - \gamma_{\dK})\|_{\dK}.
\end{align*}
Invoking the discrete trace inequality \eqref{discrete} and the inverse inequality~\eqref{eq:discrete inverse} with $p:=k+2$ yields
\begin{align*}
\|\hes&( R_T^{k+2} (\widehat{v}_T) -v_T)\|_{T}^2 \lesssim
\frac{(k+2)^6}{h^3_T}\|v_{\dK} - v_T\|^2_{\dK} + \frac{(k+2)^2}{h_T}  \|\gamma_{\dK}- \Pi^{k}_{\dK}(\partial_n v_T)\|^2_{\dK}.
\end{align*}
Recalling \eqref{def:stab} completes the proof.
\end{proof}

\begin{remark}[$d=2$] \label{rem:2d_hho}
We point out that, for $d=2$, it is possible to define the HHO method with the second component $v_{\dK} \in \mathbb{P}^{k+1}(\FK)$ instead of $\mathbb{P}^{k+2}(\FK)$. This is the HHO(A) method defined in \cite{DongErn2021biharmonic}. The present analysis also applies to this variant; see Remark~\ref{rem:2D}.
\end{remark}


\subsection{Discrete problem}

The global HHO space is defined as
\begin{equation}
\widehat{V}^k_h:=  \mathbb{P}^{k+2} (\mesh)  \times  \mathbb{P}^{k+2} (\Fall)  \times \mathbb{P}^{k} (\Fall).
\end{equation}
A generic element $\widehat{v}_h \in \widehat{V}^k_h$ is denoted as $\widehat{v}_h := \big( v_{\mesh},  v_{\Fall}, \gamma_{\Fall} \big)$ with  $v_{\mesh} := (v_T)_{T \in \mesh}$, $v_{\Fall} := (v_F)_{F \in \Fall}$, and $\gamma_{\Fall} := (\gamma_F)_{F \in \Fall}$. For every mesh cell $T \in \mesh,$  the local components of $\widehat{v}_h$ are
\begin{equation} \label{eq:localization}
\widehat{v}_T := \big( v_T, v_{\dK}:=(v_F)_{F\in\FK}, \gamma_{\dK}:=((\b{n}_F{\cdot}\b{n}_{T}) \gamma_{F})_{F\in\FK} \big) \in \widehat{V}_T^k.
\end{equation}
Notice that the localization of the third component takes into account the relative orientation of $F$ and $T$ through the factor $\n_F{\cdot}\n_T=\pm1$.
Homogeneous boundary conditions are enforced strongly by considering the subspace
\begin{equation}
\widehat{V}^k_{h0}:= \big\{ \widehat{v}_h \in \widehat{V}^k_h\;|\;v_F \equiv 0, \; \gamma_F \equiv {0}, \; \forall F \in \cF^b_h\big\}.
\end{equation}

The discrete HHO biharmonic problem reads as follows: Find $\widehat{u}_h \in \Vhkz$ such that
\begin{equation}\label{discrete_prob}
a_h(\widehat{u}_h, \widehat{w}_h) = (f, w_{\mesh}) \qquad \forall \widehat{w}_h \in \Vhkz,
\end{equation}
with the global discrete bilinear form $a_h(\widehat{v}_h, \widehat{w}_h) := \sum_{T \in \mesh}a_T(\widehat{v}_T,  \widehat{w}_T)$ where
\begin{equation}
a_T( \widehat{v}_T, \widehat{w}_T):= (\hes  R_T^{k+2} (\widehat{v}_T),  \hes  R_T^{k+2} (\widehat{w}_T))_{T} + S_{\dK} (\widehat{v}_T,  \widehat{w}_T).
\end{equation}
As shown in \cite[Lemma~4.1]{DongErn2021biharmonic}, the discrete problem~\eqref{discrete_prob} is well-posed. Moreover, \eqref{discrete_prob} is amenable to static condensation, i.e., the cell unknowns can be locally eliminated within each mesh cell, yielding a global linear system coupling only the face unknowns in $\mathbb{P}^{k+2}(\Fall) \times \mathbb{P}^{k}(\Fall)$.


\subsection{Local interpolation operators} \label{sec:interp_HHO}

In this section, we consider two interpolation (reduction) operators that can be used to map the exact solution to the HHO space.
The first operator can be defined locally using the canonical hybrid finite element interpolation operator (see, e.g., \cite[Section 7.6]{Ern_Guermond_FEs_I_2021}). Its construction is dimension-dependent. Its main advantage is that it enjoys a remarkable property when composed with the HHO reconstruction operator (see Lemma~\ref{eq:bound:elliptic-ptojection} below). The downside is that its $hp$-approximation properties have not yet been studied analytically. In what follows, we invoke~\eqref{eq: assumption} below, which we state as an assumption for which we provide numerical verifications for $d\in\{2,3\}$ in Section~\ref{valid:example1}.
The second interpolation operator is defined globally using the modified Babu\v{s}ka--Suri $hp$-interpolation operator. Its needed $hp$-approximation properties are available (see Corollary~\ref{cor: global interpolation of BS}), but its use in the a posteriori error analysis leads to additional terms in the upper bound because this interpolation operator does not combine with the HHO reconstruction operator as nicely as the first interpolation operator.

Let us now give some details.
On every mesh cell $T\in\mesh$, the local
interpolation operator associated with the canonical hybrid finite element,
$\cC_{T}^{k+2}: H^2(T)\rightarrow \mathbb{P}^{k+2}(T)$ is specified by prescribing its degrees of freedom (assuming $d=3$ and recalling that $k\ge0$) as follows:
\begin{subequations} \label{eq:def_canonical_FEM} \begin{alignat}{2}
\label{vertex property}
\cC_{T}^{k+2}(v)(\b{a}) & =  v(\b{a}), &\qquad& \forall \b{a} \in \mathcal{V}_T, \\
\label{edge property}
(\cC_{T}^{k+2}(v), \xi_E)_{E} &=  (v, \xi_E)_{E}, &\qquad& \forall \xi_E \in\mathbb{P}^{k}(E), \forall E \in \mathcal{E}_T,\\
\label{face property}
(\cC_{T}^{k+2}(v), \xi_F)_{F} &=  (v, \xi_F)_{F}, &\qquad& \forall \xi_F \in\mathbb{P}^{k-1}(F), \forall F \in \mathcal{F}_T,\; k\ge1,\\
\label{cell property}
(\cC_{T}^{k+2}(v), \xi_T)_{T} &=  (v, \xi_T)_{T}, &\qquad& \forall \xi_T \in\mathbb{P}^{k-2}(T), \; k\geq 2,
\end{alignat} \end{subequations}
where $\mathcal{V}_T$, $\mathcal{E}_T$, and $\mathcal{F}_T$ collect the vertices, edges, and faces of $T$, respectively. For $d=2$, the operator $\cC_{T}^{k+2}$ is defined in an analogous way by employing \eqref{vertex property}, \eqref{edge property}, and \eqref{cell property} with $\xi_T \in \mathbb{P}^{k-1}(T)$, $k \geq 1$.
The operator $\cC_{T}^{k+2}$ can be used to define the interpolation (reduction) operator $\hat{\mathcal{I}}_{T}^k : H^2(T) \rightarrow \widehat{V}^k_T$ such that, for all $v\in H^2(T)$,
\begin{equation} \label{eq:local_inter_I}
\hat{\mathcal{I}}_{T}^k (v) := \big(\cC_T^{k+2}(v),\cC_T^{k+2}(v)|_{\dK},\Pi_{\dK}^k( \boldsymbol{n}_T {\cdot} \nabla v)\big).
\end{equation}

\begin{lemma}[$H^2$-elliptic projection]\label{eq:bound:elliptic-ptojection}
The following holds for all $v\in H^2(T)$ and all $T\in\mesh$:
\begin{equation} \label{eq:H2_ell_proj}
(\hes R^{k+2}_T (\hat{\mathcal{I}}_{T}^k (v)) , \hes w)_{T} = (\hes v, \hes w)_{T}
\qquad \forall w \in \mathbb{P}^{k+2}(T).
\end{equation}
\end{lemma}

\begin{proof}
A proof can be found in \cite[Lemma~4.9]{DE23}.
\end{proof}

Since the $hp$-approximation properties of the interpolation operator $\cC_T^{k+2}$ have not yet been studied analytically, we make the following assumption which is supported by numerical experiments in Subsection~\ref{valid:example1}.

\begin{assumption}\label{assumption 1}
The following holds for all $v \in H^2(T)$ and all $T \in \mesh$:
\begin{align}\label{eq: assumption}
\| \nabla(v- \cC_T^{k+2} (v) )\|_{\dK}^2
\lesssim\, \hbar_T \|\hes v\|_{T}^2.
\end{align}
\end{assumption}

We now introduce two global HHO interpolation (reduction) operators $\hat{\mathcal{I}}_{h}^k: H^1_0(\Omega) \cap H^2(\Omega) \rightarrow \Vhkz$ and $\hat{\mathcal{J}}_{h}^k: H^1_0(\Omega) \cap H^2(\Omega) \rightarrow \Vhkz$, such that, for all $v \in H^1_0(\Omega)\cap H^2(\Omega)$,
\begin{subequations} \begin{align}
\hat{\mathcal{I}}_{h}^k (v) &:= \big(\cC_h^{k+2}(v), (\cC_h^{k+2}(v)|_F)_{F\in\Fall}, (\Pi_{F}^k( \boldsymbol{n}_F {\cdot}(\nabla v)|_F))_{F\in \Fall}\big), \label{def: global HHO interpolation CH}\\
\hat{\mathcal{J}}_{h}^k (v) &:= \big(\mBS^{k+2}(v), (\mBS^{k+2}(v)|_F)_{F\in\Fall}, (\Pi_{F}^k( \boldsymbol{n}_F {\cdot}(\nabla v)|_F))_{F\in \Fall}\big). \label{def: global HHO interpolation BS}
\end{align} \end{subequations}
Both interpolation operators enjoy two remarkable properties: their first component sits in $H^1_0(\Omega)$ and their second component is the trace on the mesh skeleton of its first component. We use the localization mechanism described in~\eqref{eq:localization} to define the local components of these interpolation operators attached to a generic mesh cell $T\in\mesh$. This leads to the triple $\hat{\mathcal{I}}_{T}^k (v)$ (which coincides with~\eqref{eq:local_inter_I}) and the triple $\hat{\mathcal{J}}_{T}^k (v) := \big(\mBS^{k+2}(v)|_T, (\mBS^{k+2}(v)|_F)_{F\in\FK}, (\Pi_{F}^k( \boldsymbol{n}_T {\cdot} (\nabla v)_F ))_{F\in\FK}\big)$.

\begin{lemma}[Bound on stabilization] \label{lem:bound_stab}
\textup{(i)} Under Assumption \ref{assumption 1}, the following holds for all $v\in H^2(T)$ and all $T\in\mesh$:
\begin{subequations}
\begin{equation}\label{Stabilisation of reduction bound: CH}
S_{\dK}(\widehat{\mathcal{I}}_{T}^k(v), \widehat{\mathcal{I}}_{T}^k(v))\lesssim (k+2)\|\hes v\|^2_{T}.
\end{equation}
\textup{(ii)} The following holds for all $v\in H^1_0(\Omega)\cap H^2(\Omega)$ and all $T\in\mesh$:
\begin{equation} \label{Stabilisation of reduction bound: BS}
S_{\dK}(\widehat{\mathcal{J}}_{T}^k(v), \widehat{\mathcal{J}}_{T}^k(v))\lesssim   (k+2)\|\hes v\|^2_{T}.
\end{equation}
\end{subequations}
\end{lemma}
\begin{proof}
Since the second component of $\widehat{\mathcal{I}}_{T}^k(v)$
is the trace on $\dK$ of its first component, we infer that
\begin{equation*}
S_{\dK} (\widehat{\mathcal{I}}_T^k(v), \widehat{\mathcal{I}}_T^k(v)) =
(k+2) \hbar_T^{-1} \|\Pi_{\dK}^{k}\big(\partial_nv - \partial_n {\cC}_T^{k+2}(v)\big)\|^2_{\dK}.
\end{equation*}
Using the $L^2$-stability of $\Pi^k_{\dK}$ and invoking Assumption \ref{assumption 1} proves~\eqref{Stabilisation of reduction bound: CH}. The proof of~\eqref{Stabilisation of reduction bound: BS} is similar, but does not need to invoke any assumption as the $hp$-approximation estimate is available from Corollary~\ref{cor: global interpolation of BS}.
\end{proof}

\section{Local Helmholtz decomposition}\label{sec: Local Helmholtz Decomposition}

In this section, we derive a Helmholtz decomposition of tensor-valued fields in $\ut{\b{L}}^2(\oma)$ where $\oma$ is the patch associated with the generic mesh vertex $\ba\in\vertice$. We consider the case $d=3$ (the case $d=2$ is analogous and simpler). The key novelty is to show that the stability constant only depends on the mesh shape-regularity.

We first clarify several geometric properties of vertex patches. Since the mesh is shape-regular, the number of simplices sharing a vertex is uniformly bounded by a constant depending only on the shape-regularity parameter of the mesh and the space dimension. 
Moreover, if $\ba$ is an interior vertex, then $\oma$ is homeomorphic to an open ball, whereas, if $\ba$ is a boundary vertex, $\oma$ is homeomorphic to a half open ball. Therefore, every vertex patch, including boundary vertex patches, is a connected, simply connected Lipschitz domain with a connected boundary.
Moreover, for every interior vertex, $\oma$ is star-shaped with respect to a ball whose radius is comparable to the diameter of $\oma$, with a constant depending only on the mesh shape-regularity; see \cite[Proposition~8.2]{chaumontfrelet:hal-05204325}, building on the two-dimensional result of \cite{LeePre:79}. For every boundary vertex, uniform star-shapedness with respect to a single ball may fail, for example near re-entrant corners. In this case, $\oma$ can be decomposed into a uniformly bounded chain of star-shaped subdomains formed by unions of simplices sharing a common face; see the proof of \cite[Lemma~4.3]{Dong-Ern-DG:2026}.
An important consequence of these geometric properties is that there is a right inverse of the curl operator for divergence-free vector fields on vertex patches whose stability constant depends only on the mesh shape-regularity. This follows from \cite[Corollary~29]{GuzSal21} for interior vertex patches and from \cite[Theorem~35]{GuzSal21} for boundary vertex patches. This result will play a key role in establishing the stability estimate for the local Helmholtz decomposition.

Let $\varepsilon_{ijk}$ denote the Levi-Civita symbol for all $i,j,k\in\{1{:}d\}$.
We define the skew-symmetric tensor-valued operator $\sk$ such that $\sk(\b{v})_{ij}:=\varepsilon_{ijk} v_{k}$ for all $\bv:=(v_k)_{k\in\{1{:}d\}}$. We also recall that the (vector-valued) curl of $\bv$ has components $(\nabla {\times}\bv)_{i} := \varepsilon_{ijk} \partial_j v_{k}$ for all $i\in\{1{:}d\}$. Here and in what follows, we employ the usual summation convention on repeated indices. For a tensor-valued field $\ut{\b{A}} = (A_{ij})_{i,j\in\{1{:}d\}}$, its (tensor-valued) row-wise curl is defined as $(\curlrw{\b{\ut{A}}})_{ij}:= \varepsilon_{jkl} \partial_k A_{il}$ for all $i,j\in\{1{:}d\}$, and its (vector-valued) row-wise divergence is defined as $(\divrw{\ut{\b{A}}})_{i} = \partial_j A_{ij}$ for all $i\in\{1{:}d\}$. Notice that $\divrw(\curlrw{\b{\ut{A}}})=\boldsymbol{0}$.

\begin{lemma}[Local  Helmholtz decomposition] \label{lemma:Helmholtz-dec}
For all $\ba \in \vertice$ and all $\ut{\b{\Sigma}}\in \ut{\b{L}}^2(\oma)$,
there exist a (scalar-valued) function $\xi \in H^2_0(\oma)$, a (vector-valued) field $\b{\rho} \in \boldsymbol{H}_0(\mathrm{div}=0;\oma)$ (i.e., $\nabla {\cdot} \b{\rho} = 0$ in $\oma$ and $\b{\rho}{\cdot} \b{n}_{\oma}=0 $ on $\partial \oma$), and a (tensor-valued) field  $\ut{\b{\Psi}} \in \ut{\b{H}}^1(\oma)$, so that the following holds:
\begin{subequations} \begin{equation} \label{def:Helmholtz}
\ut{\b{\Sigma}} = \hes \xi + \sk(\b{\rho}) + \curlrw{\b{\ut{\Psi}}},
\end{equation}
with\begin{equation} \label{bound:Helmholtz}
\Vert \hes \xi \Vert _{\omega_{\ba}} + \Vert {\b{\rho}} \Vert_{\omega_{\ba}} + |\ut{\b{\Psi}} |_{\ut{\b{H}}^1(\omega_{\ba})} \lesssim \Vert \ut{\b{\Sigma}} \Vert_{\omega_{\ba}}.
\end{equation}
\end{subequations}
\end{lemma}
\begin{proof}
(1) Let $\xi\in H^2_0(\oma)$ be such that
\begin{equation} \label{eq:def_xi}
(\hes \xi, \hes v)_{\oma} = (\ut{\b{\Sigma}} , \hes v)_{\oma}
\qquad \forall v\in H^2_0(\oma).
\end{equation}
The Cauchy--Schwarz inequality gives
\begin{equation}\label{local HD bound 1}
\|\hes \xi\|_{\oma} \leq  \| \ut{\b{\Sigma}}\|_{\oma}.
\end{equation}
Moreover, considering arbitrary test functions $v\in C_0^\infty(\oma)$ in~\eqref{eq:def_xi} implies that
\begin{equation} \label{eq:divrw=0}
\nabla {\cdot} (\divrw{(\hes \xi - \ut{\b{\Sigma}})}) =0.
\end{equation}
Moreover, since $(\hes \xi - \ut{\b{\Sigma}})\in \ut{\b{L}}^2(\oma)$, we infer that
\begin{align*}
\| \divrw {(\hes \xi - \ut{\b{\Sigma}})}\|_{\b{H}^{-1}(\oma)}
&=   \sup_{\b{v} \in \b{H}^1_0(\oma )}  \frac{|\langle \divrw{ (\hes \xi - \ut{\b{\Sigma}})},\b{v}\rangle_{\oma}|}{\|\nabla \b{v}\|_{\oma}} \\
&
 =  \sup_{\b{v} \in \b{H}^1_0(\oma)}  \frac{|(\hes \xi - \ut{\b{\Sigma}}, \nabla \b{v})_{\oma}|}{\|\nabla \b{v}\|_{\oma}}
\leq  \|\hes \xi - \ut{\b{\Sigma}}\|_{\oma} \leq 2\|\ut{\b{\Sigma}}\|_{\oma},
\end{align*}
where $\langle \cdot, \cdot \rangle_{\oma}$ denotes the duality pairing between $\b{H}^{-1}(\oma)$ and $\b{H}^1_0(\oma)$ and where the last bound follows from \eqref{local HD bound 1}.

(2) Owing to~\eqref{eq:divrw=0} and since $\oma$ is simply connected with a connected boundary, we infer from \cite[Section 3]{Amrouche07} that there exists vector field $\b{\rho} \in \boldsymbol{H}_0(\mathrm{div}=0;\oma)$ such that $\nabla {\times} \b{\rho} = -\divrw {(\hes \xi - \ut{\b{\Sigma}})}$. This identity can be rewritten as
\begin{equation}\label{local HD bound 5}
\divrw( \ut{\b{\Sigma}} - \sk(\b{\rho}) -\hes \xi ) = \b{0}.
\end{equation}
Moreover, we observe that
\begin{equation*}
\| \nabla {\times}  \b{\rho}\|_{\b{H}^{-1}(\oma)}
= \underset{{\b{v} \in \b{H}^1_0(\oma)}}{\sup} \frac{|\langle \nabla {\times}  \b{\rho},\b{v}\rangle_{\oma}|}{\|\nabla \b{v}\|_{\oma}}
= \underset{{\b{v} \in \b{H}^1_0(\oma)}}{\sup} \frac{|( \b{\rho},  \nabla {\times} \b{v})_{\oma}|}{\|\nabla \b{v}\|_{\oma}}.
\end{equation*}
Owing to \cite[Theorem 35 and Corollary 29]{GuzSal21}, there exists $\b{w} \in \b{H}^1_0(\oma)$ such that $\nabla {\times} \b{w}= \b{\rho}$ and $\|\nabla \b{w} \|_{\oma} \lesssim \|\nabla {\times }\b{w}\|_{\oma} = \|\b{\rho}\|_{\oma}.$
Combining the above bounds, we obtain
\begin{align}\label{local HD bound 3}
\|\b{\rho}\|_{\oma} \lesssim   \frac{|( \b{\rho},  \nabla {\times} \b{w})_{\oma}|}{\|\nabla \b{w}\|_{\oma}}
&\leq
\underset{{\b{v} \in \b{H}^1_0(\oma)}}{\sup} \!\!\frac{|( \b{\rho},  \nabla {\times} \b{v})_{\oma}|}{\|\nabla \b{v}\|_{\oma}}
 \nonumber \\
&= \|\nabla {\times}  \b{\rho}\|_{\b{H}^{-1}(\oma)}
=  \| \divrw {(\hes \xi - \ut{\b{\Sigma}})}\|_{\b{H}^{-1}(\oma)}
 \leq 2\|\ut{\b{\Sigma}}\|_{\oma}.
\end{align}

(3) Owing to~\eqref{local HD bound 5} and invoking \cite[Corollary 29 \& Theorem 35]{GuzSal21} for each
row of $\ut{\b{\Sigma}} - \sk(\b{\rho}) -\hes \xi$, we infer that there exists
$\ut{\b{\Psi}} \in \ut{\b{H}}^1(\oma)$ so that
\begin{equation*}
\curlrw{\b{\ut{\Psi}}} = \ut{\b{\Sigma}} - \sk(\b{\rho}) -\hes \xi,
\qquad
|\ut{\b{\Psi}} |_{\ut{\b{H}}^1(\omega_{\ba})} \lesssim \|\ut{\b{\Sigma}} - \sk(\b{\rho}) -\hes \xi\|_{\oma}.
\end{equation*}
The above identity is nothing but~\eqref{def:Helmholtz}. Moreover, the above
bound combined with the triangle inequality and the bounds~\eqref{local HD bound 1}
(on $\hes\xi$) and \eqref{local HD bound 3} (on $\b{\rho}$) leads to~\eqref{bound:Helmholtz}. This completes the proof.
\end{proof}
\begin{remark}[General boundary conditions]
The above local Helmholtz decomposition is suitable for the analysis of biharmonic problems with essential boundary conditions. For more general boundary conditions, such as free, simply supported, or mixed boundary conditions, one needs to establish a corresponding local Helmholtz decomposition adapted to these settings. We refer to \cite{BeiraoNiiranenStenberg10} for a Helmholtz decomposition under general boundary conditions on two-dimensional simply connected domains. The extension of the present analysis to more general boundary conditions is left to future work.
\end{remark}

\section{$hp$-a posteriori error analysis}\label{sec:apost}

In this section, we carry out the residual-based $hp$–a posteriori error analysis for the HHO discretization of the biharmonic problem. 

Our goal is to establish an upper bound on the approximation error
\begin{equation}
e:= u - u_{\mesh},
\end{equation}
which measures the difference between the exact solution of~\eqref{weakform}, $u$, and the cellwise component of the HHO solution of~\eqref{discrete_prob}, $u_{\mesh}$.
We will use the following local error indicators: For all $T\in\mesh$,
\begin{subequations} \label{Error indicators} \begin{align}
\eta_{T,{\rm sta}}
:={}&  S_{\dK}(\hat{u}_T,\hat{u}_T)^{\frac12}, \label{eq:def_eta_sta} \\
\eta_{T,{\rm res}}
:={}& \hbar_T^2 \|\Pi^{k-2}_T (f) - \Delta^2  R^{k+2}_T (\widehat{u}_T)\|_{T} + \hbar_T^{\frac12} \|\sjump{\partial_{nt}  R^{k+2}_T (\widehat{u}_T)}\|_{\dKi} \nonumber \\
&+ \hbar_T^{\frac32}\|\sjump{\partial_n \Delta  R^{k+2}_T (\widehat{u}_T)}\|_{\dKi}, \label{eq:def_eta_res}\\
\eta_{T, {\rm tan}}
:={}& \hbar_T^{\frac12} \Big\{ \|\sjump{\partial_{tt} u_{\mesh}}\|_{\dK}
+
\|\sjump{\partial_{nt} u_{\mesh}}\|_{\dK}\Big\}, \label{eq:def_eta_tan}
\end{align}
recalling the convention that $\Pi^{n}_{T}(f)=0$ for all $n\le -1$, and
where jumps of normal derivatives across any mesh interface $F\in\Fint$ are understood so that the normal derivative is taken along the normal vector $\b{n}_F$ orienting $F$ on both sides of $F$. We define the data oscillation term
\begin{equation} \label{est:oscillation}
\mathcal{O}(f)_T
:= \hbar_T^2 \| f - \Pi^{k-2}_{T}(f)\|_{T}.
\end{equation}
\end{subequations}

\subsection{Abstract error bound}

Our first step is to bound the error as the sum of the dual norm of a suitable residual functional on the energy space $H^2_0(\Omega)$ and the nonconforming error measuring the departure of $u_{\mesh}$ from this space.

\begin{lemma}[Abstract error bound]\label{Lemma: abstract error bound}
The following holds:
\begin{equation}\label{eq: abstract error bound}
\|\hes_{\mesh} e\|_{\Omega}^2
 =  \|\mathcal{R}_{\mesh}\|_{H^{-2}(\Omega)}^2
+\min_{w\in H^2_0(\Omega)}\|\hes_{\mesh} (w - u_{\mesh}) \|_{\Omega}^2,
\end{equation}
where $\mathcal{R}_{\mesh}\in H^{-2}(\Omega)$ denotes the residual functional such that
\begin{equation}\label{def: residual functional}
\mathcal{R}_{\mesh}(w) := (\hes_{\mesh} e, \hes w)_{\Omega}
= (f,w)_\Omega - (\hes_{\mesh} u_{\mesh}\hes w)_{\Omega}, \qquad \forall w\in H^2_0(\Omega),
\end{equation}
and its dual norm is defined as
\begin{equation} \label{eq:def_res_dual}
\|\mathcal{R}_{\mesh}\|_{H^{-2}(\Omega)} := \sup_{w\in H^2_0(\Omega)} \frac{\mathcal{R}_{\mesh}(w)}{\|\hes w\|_\Omega}.
\end{equation}
\end{lemma}
\begin{proof}
The proof can be found in \cite[Thereom 1]{CarstenGallistlHu13}.
\end{proof}

%

\subsection{Bound on dual residual norm}

To bound the dual residual norm, we can invoke any of the two HHO interpolation operators introduced in Section~\ref{sec:interp_HHO}. This is the reason why the bound on the dual residual norm takes the form of a minimum between two terms.

\begin{lemma}[Bound on dual residual norm]\label{lemma: dual norm of residual}
The following holds:
\begin{subequations}
\begin{equation} \label{eq:dual_res1}
\|\mathcal{R}_{\mesh}\|_{H^{-2}(\Omega)}^2 \lesssim \sum_{T\in\mesh} \big\{
(k+2) \eta_{T,{\rm sta}}^2 + \eta_{T,{\rm res}}^2+ \mathcal{O}(f)^2_T\big\}.
\end{equation}
Moreover, under Assumption \ref{assumption 1}, the following holds:
\begin{align}
\|\mathcal{R}_{\mesh}\|_{H^{-2}(\Omega)}^2
\lesssim {}&\!\sum_{T\in\mesh}\! (k+2) \eta_{T,{\rm sta}}^2 
+ \min\bigg\{\!\sum_{T\in\mesh}\!\!
\big\{\eta_{T,{\rm res}}^2+ \mathcal{O}(f)^2_T\big\},
\!\sum_{T\in\mesh}\!\! (k+2) \mathcal{O}(f)^2_T\bigg\}. \label{eq:dual_res2}
\end{align} \end{subequations}
\end{lemma}
\begin{proof}
The proof is postponed to Section \ref{sec: proof dual norm of residual}.
\end{proof}

\subsection{Bound on nonconforming error}

To bound the nonconforming error, it suffices to pick any function in $H^2_0(\Omega)$. In this section, we show how to reconstruct from $u_{\mesh}$ such a function using a $C^1$-partition of unity and local solves on vertex patches. To this purpose, we use the $C^1$-composite finite element with polynomial order five on the Alfeld split of the mesh $\mesh$ introduced in \cite{Walk14}. Recall that the Alfeld split consists in subdividing each mesh cell $T\in\mesh$ into $(d+1)=4$ sub-simplices by connecting each vertex of $T$ to its centroid. We denote a generic sub-cell obtained in this manner by $\tilde{T}$, and the global resulting sub-mesh by $\Amesh$. On each macro-cell $T$, the composite element consists of piecewise quintic polynomials defined on the Alfeld sub-simplices $\tilde{T}$ and that are globally $\mathcal{C}^1$-continuous over $T$. The associated degrees of freedom include, in particular, the values at the vertices of $T$ and at its centroid $\bc_T$. The remaining degrees of freedom correspond to derivatives and are not used in the present construction. In what follows, we denote by $\{\psi_{\ba}\}_{\ba\in\vertice}$ the global basis functions associated with the vertex-based  degrees of freedom of the original mesh $\mesh$, and by $\{\psi_T\}_{T\in\mesh}$ the global basis functions associated with the cell-based degrees of freedom of the original mesh. The support of each vertex-based basis function $\psi_{\ba}$ is the corresponding vertex patch $\oma$, whereas the support of each cell-based basis function $\psi_T$ is $T$. A key property of the construction in \cite{Walk14} is the following partition-of-unity identity:
\begin{equation}\label{eq:pou}
\sum_{\ba\in\mathcal{V}_h}\psi_{\ba}
+\sum_{T\in\mesh}\psi_T
\equiv 1
\qquad\text{on }\Omega.
\end{equation}
Indeed, the other basis functions do not appear in \eqref{eq:pou}, since their associated degrees of freedom involve normal derivatives, which vanish when applied to the constant function $1$.


\begin{definition}[Patchwise $H^2$-potential reconstruction]\label{def:H2-potential reconstruction}
Let $u_{\mesh}\in \mathbb{P}^{k+2}(\mesh)$.
For all $ \ba \in \mathcal{V}_h $, let $ u_c^{\ba} \in H^2_{0}(\oma) $ be the solution to the following well-posed problem:
\begin{align}\label{def: vertex reconstruction}
(\hes u_c^{\ba}, \hes v)_{\oma} = (\hes_{\mesh} (\psi_{\ba} u_{\mesh}), \hes v)_{\oma} \qquad \forall v \in H^2_{0}(\oma).
\end{align}
Equivalently,
\begin{align}
u_c^{\ba} &:= \argmin_{\rho_{\ba} \in H^2_{0}(\oma)} \| \hes \rho_{\ba} - \hes_{\mesh} (\psi_{\ba} u_{\mesh}) \|_{\oma}. \label{eq:local_min}
\end{align}
Then, extending $u_c^{\ba}$ by zero to $\Omega$,  we set
\begin{equation}\label{uc:partition}
u_c := \sum_{\ba \in \mathcal{V}_h} u_c^{\ba} + \sum_{T \in \mesh} \psi_{T} u_T \in H^2_0(\Omega).
\end{equation}
\end{definition}

\begin{lemma}[Bound on nonconforming error]\label{lemma: nonconforming error total}
The following holds:
\begin{align}\label{full bound for nonconforming error}
\|\hes_{\mesh} (u_c - u_{\mesh})\|^2_{\Omega} &\lesssim \sum_{T \in \mesh} \big\{\eta_{T,{\rm sta}}^2
+ \eta_{T, {\rm tan}}^2 \big\}.
\end{align}
\end{lemma}
\begin{proof}
The proof is postponed to Section \ref{sec: proof full bound}.
\end{proof}


\begin{remark}[Inhomogeneous boundary condition]\label{remark: Inhomogeneous BC}
The above $H^2$-reconstruction can be modified to account for inhomogeneous boundary conditions of the form $u = g_{\rm D}$ and  $\partial_{n} u = g_{\rm N}$ on $\partial \Omega$.
For every mesh vertex $\ba \in \mathcal{V}_h $, we then solve for $u_c^{\ba} \in H^2_{g} (\oma):=\{v\in H^2(\oma) \;|\; v|_{\partial \oma \cap \partial \Omega} = \psi_{\ba}g_{\rm D},\;  \partial_{n} v|_{\partial \oma \cap \partial \Omega} = \psi_{\ba}g_{\rm N}, \; v|_{\partial \oma \cap \Omega} = \partial_{n} v|_{\partial \oma \cap  \Omega} = 0 \}$ such that
$$
(\hes u_c^{\ba}, \hes v)_{\oma} = (\hes_{\mesh} (\psi_{\ba} u_{\mesh}), \hes v)_{\oma} \qquad \forall v \in H^2_{0}(\oma).
$$
Then, extending $u_c^{\ba}$ by zero to $\Omega$, $u_c \in H^2(\Omega)$ is still defined as in~\eqref{uc:partition}, and the partition-of-unity property~\eqref{eq:pou} implies that $u_c$ also satisfies the inhomogeneous boundary conditions. Finally, the upper bound from Lemma~\ref{lemma: nonconforming error total} still holds provided the tangential jumps defining $\eta_{T, {\rm tan}}$ are redefined as follows: For all $F\in\Fb$,
$$
\sjump{\partial_{tt} u_{\mesh}}|_F :=
\partial_{tt} \big(u_{\mesh}|_F  - E_F(g_{\rm D}|_F)\big),
\qquad
\sjump{\partial_{nt} u_{\mesh}}|_F :=
\partial_{t} \big((\partial_n u_{\mesh})|_F  - E_F(g_{\rm N}|_F)\big),
$$
where, for a function $\varphi$ defined on $F$, $E_F(\varphi)(\b{x}):=\varphi((\utI-\b{n}_F\otimes\b{n}_F)(\b{x}-\b{x}_F))$ and $\b{x}_F$ is the centroid of $F$.
\end{remark}

\subsection{Main result}
We are now ready to state our main result.

\begin{theorem}[$hp$-upper error bounds]\label{thm:upperbound}
The following holds:
\begin{subequations}
\begin{equation}
\su \!\! \big\{\|\hes_{\mesh} e\|_{T}^2 + S_{\dK}(\hat{u}_T,\hat{u}_T) \big\}
\lesssim \su \big\{ \eta_{T,\rm{tan}}^2 + (k+2)\eta_{T,{\rm sta}}^2 + \eta_{T,{\rm res}}^2+ \mathcal{O}^2_T(f)\big\}.
\end{equation}
Moreover, under Assumption \ref{assumption 1}, the following holds:
\begin{align}\label{upper error bound}
\su \!\! \big\{\|\hes_{\mesh} e\|_{T}^2 + S_{\dK}(\hat{u}_T,\hat{u}_T) \big\}
\lesssim {}& \sum_{T\in\mesh}\big\{ \eta_{T,\rm{tan}}^2
+ (k+2) \eta_{T,{\rm sta}}^2  \big\} \nonumber \\
& \!\!+ \min\bigg\{ \sum_{T\in\mesh}
\big\{ \eta_{T,{\rm res}}^2+ \mathcal{O}(f)^2_T\big\},
\sum_{T\in\mesh} (k+2) \mathcal{O}(f)^2_T \bigg\}.
\end{align}
\end{subequations}
\end{theorem}
\begin{proof}
Combine Lemmas~\ref{Lemma: abstract error bound}, \ref{lemma: dual norm of residual},
and \ref{lemma: nonconforming error total} to bound the broken Hessian of the error, whereas the stabilization term on the left-hand side is trivially bounded by $\eta_{T,{\rm sta}}^2$ for all $T\in\mesh$.
\end{proof}

\begin{remark}[$d=2$] \label{rem:2D}
The a posteriori error analysis can also be performed for the two-dimensional variant of the HHO method mentioned in Remark~\ref{rem:2d_hho}. The main difference lies in the definition of the two HHO reduction operators (see~\eqref{def: global HHO interpolation CH} and \eqref{def: global HHO interpolation BS}) which are now taken to be $\cC_h^{k+1}$ and $\mBS^{k+1}$. This choice is consistent with the polynomial degrees of the HHO(A) variant, where the cell unknown belongs to $\mathbb P^{k+2}(T)$, whereas the face unknown belongs to $\mathbb P^{k+1}(\mathcal F_{T})$. Hence, the interpolation operators have to be defined with degree $(k+1)$ in order to remain compatible with the degree of the face unknowns.
Since $\cC_h^{k+1}$ delivers $L^2$-orthogonality properties to $\mathbb{P}_{k-2}(T)$ for $d=2$, we infer that the data oscillation term $\mathcal{O}(f)^2_T$ in \eqref{est:oscillation} still involves $\| f - \Pi^{k-2}_{T}(f)\|_{T}$ as for $d=3$.
\end{remark}

\section{Numerical examples} \label{sec:Numerical example}

In this section, we present numerical examples to illustrate our theoretical results.
In Section~\ref{valid:example1}, we verify Assumption \ref{assumption 1} in two and three dimensions. In Section~\ref{sec:num_smooth}, we report convergence rates for a three-dimensional smooth solution, so as to verify our main result, Theorem~\ref{thm:upperbound}. In Section~\ref{sec:num_holes}, we verify the claim on the robustness of the estimate with respect to the topology of the domain by considering a series of 2D computational domains with increasing number of holes. Finally, in Sections~\ref{sec:num_adaptive_2D} and~\ref{sec:num_adaptive_3D}, we consider an $h$-adaptive algorithm (with fixed polynomial degree) to approximate a singular solution in 2D and 3D, respectively. The development of a fully $hp$-adaptive algorithm goes beyond the present scope.

\subsection{Verification of Assumption \ref{assumption 1}}\label{valid:example1}

We verify Assumption~\ref{assumption 1} in two and three dimensions. To this end, it is sufficient to perform the verification on a reference simplex as mapping from the reference simplex to any simplex of the mesh by using the pullback by the affine geometric transformation will bring the scaling by the mesh size and the dependency of the constant on the mesh shape-regularity, but this is independent of the underlying polynomial degree. In our calculations, we consider the reference triangle $T$ with vertices $(-0.5,0)$, $(0.5,0)$, and $(-0.5,1)$ in two dimensions, and the reference tetrahedron $T$ with vertices $(0,0,-0.5)$, $(1,0,-0.5)$, $(0,1,-0.5)$, and $(0,0,0.5)$ in three dimensions. We consider the following function, prescribed in polar (resp., spherical) coordinates:
\begin{equation}
v = r^{\alpha}, \qquad \alpha > 2-\frac{d}{2}.
\end{equation}
We notice  $u \in H^{\alpha+\frac{d}{2}-\varepsilon}(T)$ for any arbitrarily small $\varepsilon > 0$.
The expected convergence rate with respect to the polynomial degree $(k+2)$ is
\begin{align}
\| \nabla (v- \cC_T^{k+2} (v) ) \|_{\partial T}
\leq C (k+2)^{-(\alpha +\frac{d-3}{2}-\epsilon)} |v|_{H^{\alpha+\frac{d}{2} -\epsilon}(T)}.
\end{align}
(Since the singularity at the origin does not lie on a vertex of the domain, the doubling of the convergence order with respect to the polynomial degree is not expected; see \cite{BabusakGuo2000}.)

Figure~\ref{Ex1:p-refienement} reports the interpolation error $\|\nabla (v- \cC_T^{k+2} (v) ) \|_{\partial T}$ for $d=2$, $k \in \{0,\ldots,22\}$, and $\alpha \in\{ 1.01, 1.51\}$ (left panel), and for $d=3$, $k \in\{0,\ldots,13\}$, and $\alpha \in\{1.01, 1.51\}$ (right panel). To achieve the needed accuracy, a composite quadrature rule is employed to evaluate the integrals involving the singular functions.
We observe that the interpolation error converges at the optimal rate $(k+2)^{-(\alpha +\frac{d-3}{2})}$ in all cases. This behavior is consistent with Assumption~\ref{assumption 1}.

\begin{figure}[!tb]
\begin{center}
\includegraphics[width=0.46\linewidth]{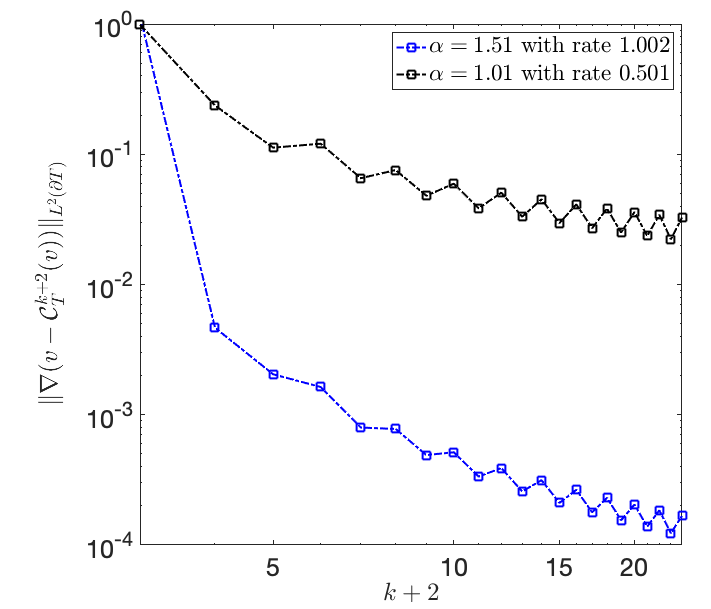}
\includegraphics[width=0.46\linewidth]{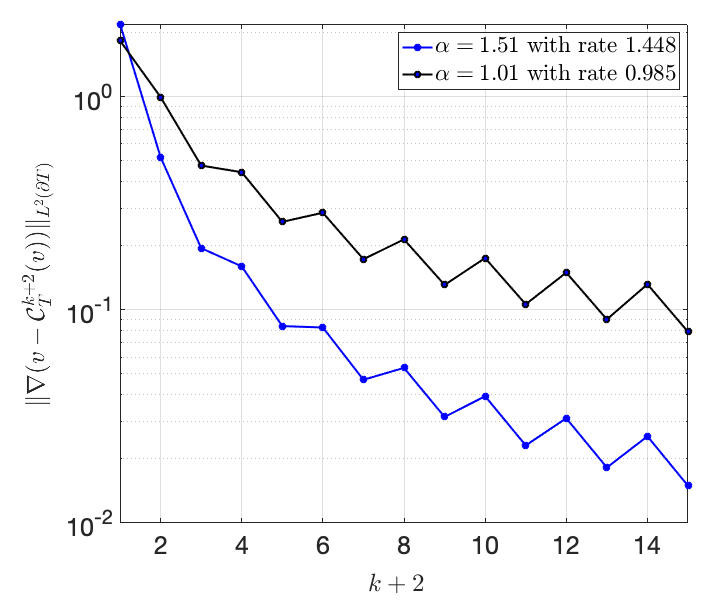}
\end{center}
\caption{Convergence of the interpolation error $\| \nabla(v- \cC_T^{k+2} (v) ) \|_{\partial T}$ as a function of the polynomial degree $(k+2)$. Left panel: $d=2$, $k \in \{0,\ldots,22\}$, and $\alpha \in\{ 1.01, 1.51\}$. Right panel: $d=3$, $k \in\{0,\ldots,13\}$, and $\alpha \in\{1.01, 1.51\}$.}\label{Ex1:p-refienement}
\end{figure}

\subsection{Example 1: Convergence rates for smooth solution in 3D}
\label{sec:num_smooth}

\begin{figure}[!tb]
\begin{center}
\begin{tabular}{cc}
\includegraphics[width=0.46\linewidth]{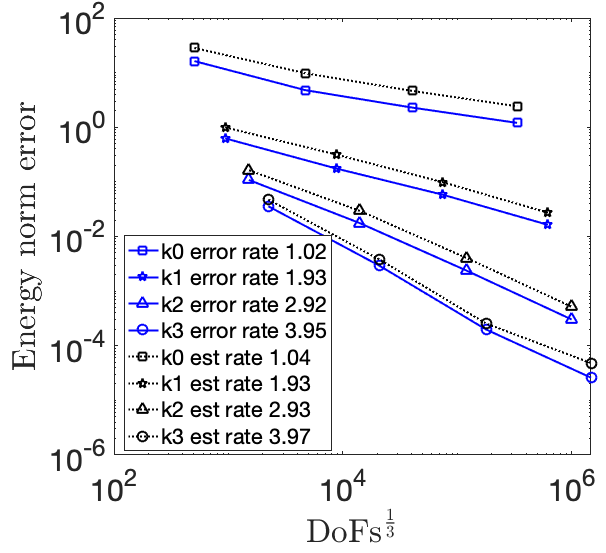} &
\includegraphics[width=0.46\linewidth]{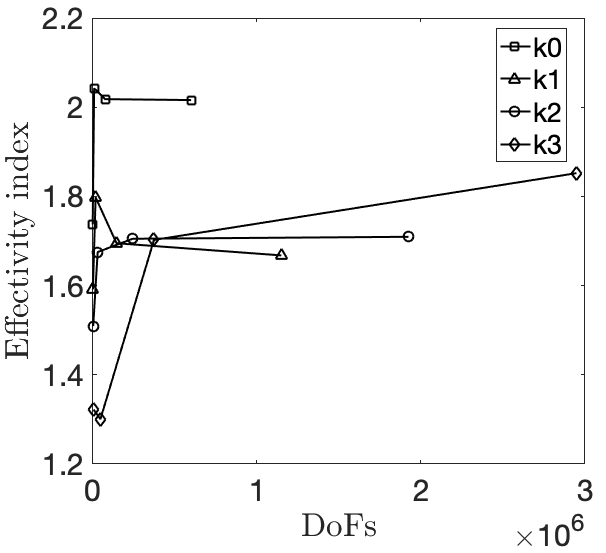}
\end{tabular}
\end{center}
\caption{Example 1: Energy error and a posteriori error estimator for $k\in\{0,1,2,3\}$ as a function of DoFs (left panel) and effectivity index as a function of DoFs (right panel).}\label{Ex2:error_uniform_refinement}
\end{figure}
In this example, we select $f$ and the boundary conditions on the unit cube $\Omega := (0,1)^3$, so that the exact solution is
\begin{equation}
u(x,y,z) := \sin(\pi x) \sin(\pi y) \sin(\pi z) x(1-x)y(1-y)z(1-z).
\end{equation}
We employ the polynomial degrees $k\in\{0,\ldots,3\}$ and a sequence of successively refined tetrahedral meshes consisting of $48$, $384$, $3072$, and $24576$ cells.

Let us first verify the convergence rates of the HHO method for $k\in\{0,\dots,3\}$.
Figure~\ref{Ex2:error_uniform_refinement} reports the energy error corresponding to the left-hand side of~\eqref{upper error bound} (left panel) and the \emph{a posteriori} error estimator corresponding to the right-hand side of~\eqref{upper error bound} (right panel). The rates are computed as a function of DoFs, which denotes the total number of globally coupled discrete unknowns (that is, the total number of face unknowns except those located on the boundary faces). We observe that the energy error and the a posteriori estimator both converge at the optimal rate $\mathcal{O}(\textup{DoFs}^{-\frac{(k+1)}{3}})$. Moreover, the right panel of Figure~\ref{Ex2:error_uniform_refinement} reports the effectivity index, defined as the ratio of the a posteriori estimator to the energy error. We observe that the effectivity index remains well behaved as a function of DoFs, taking values between 1.2 and 1.85 for $k\geq1$, whereas the effectivity index is slightly greater than 2 for $k=0$. Finally, we observe that the minimum in the upper error bound \eqref{upper error bound} is always reached by the second component involving only the data oscillation term.

\subsection{Example 2: Effectivity on 2D domains with increasing number of holes}
\label{sec:num_holes}

In this example, we investigate the robustness of the proposed estimator with respect to the topology of the computational domain. In particular, we consider computational domains $(\Omega_i)_{i\in\{1{:}4\}}$ such that $\Omega_i$ is defined as the unit square domain from which $n_i$ rectangular holes are removed, with $n_i\in\{0,1,4,8\}$. The resulting domains together with the corresponding coarsest meshes are displayed in Figure~\ref{fig:four_cases}. A sequence of successively uniformly refined meshes
is employed for each domain.
The exact solution is always  $u(x,y) := (\sin(\pi x) \sin(\pi y))^2$,
and the source term together with the inhomogeneous boundary conditions on the hole boundaries are computed accordingly. Consequently, the estimators are modified as discussed in Remark~\ref{remark: Inhomogeneous BC}.
We compute the effectivity indices for the polynomial degrees $k=0$ and $k=1$.
The results are reported in Figure~\ref{fig:effectivity-k0-k1}.
We observe that the effectivity indices remain essentially stable as the number of holes increases.

\begin{figure}[!tb]
    \centering
    \begin{subfigure}{0.24\textwidth}
        \centering
        \includegraphics[width=\linewidth]{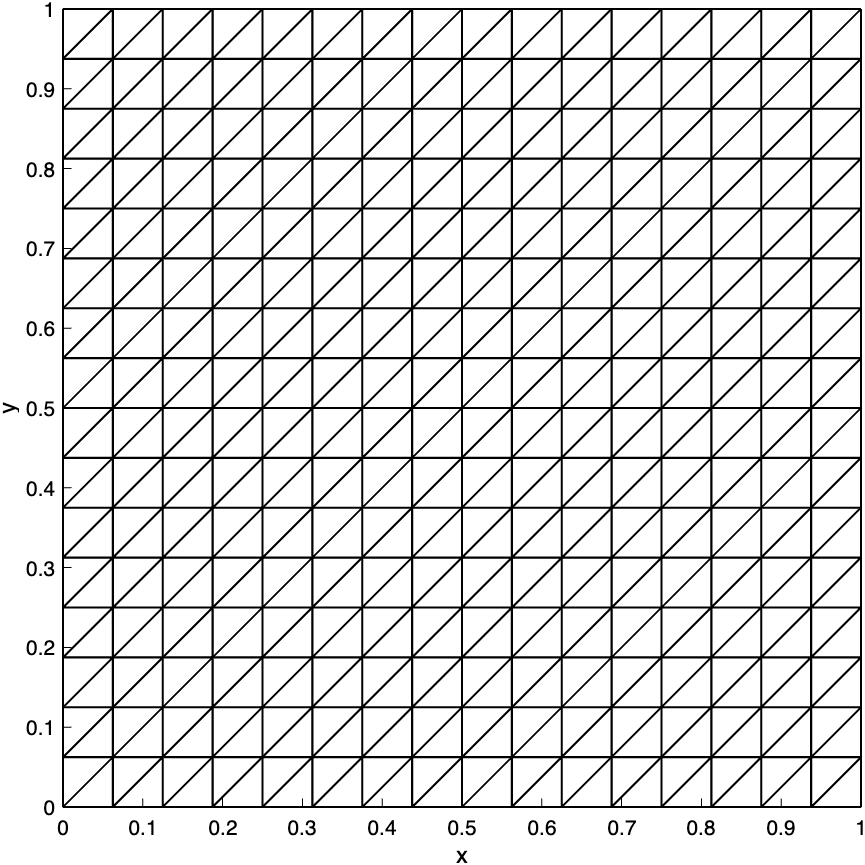}
        \caption{0 hole}
        \label{fig:first}
    \end{subfigure}
    \hfill
    \begin{subfigure}{0.24\textwidth}
        \centering
        \includegraphics[width=\linewidth]{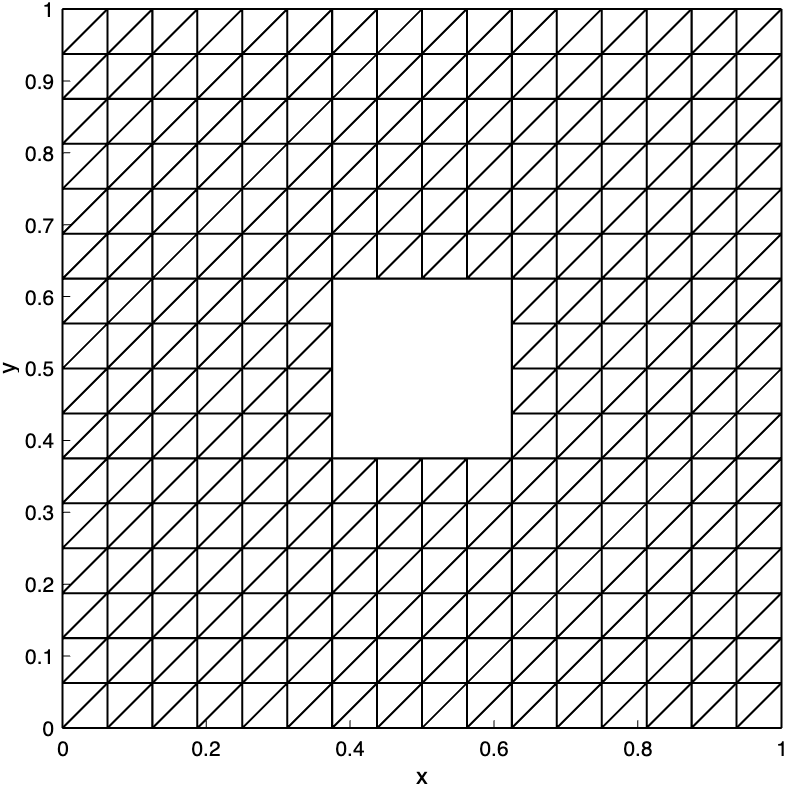}
        \caption{1 hole}
        \label{fig:second}
    \end{subfigure}
    \hfill
    \begin{subfigure}{0.24\textwidth}
        \centering
        \includegraphics[width=\linewidth]{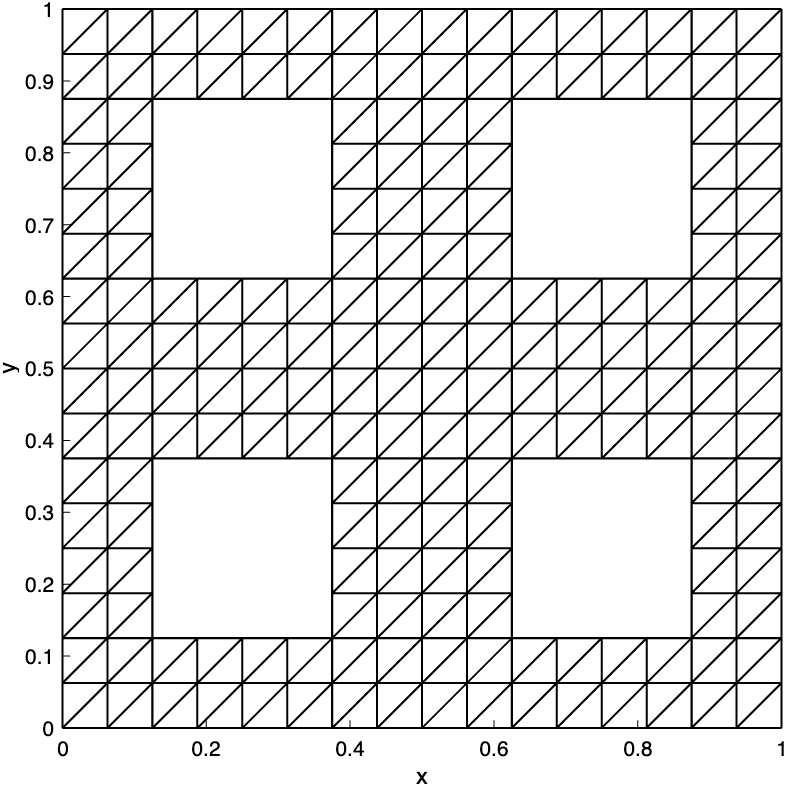}
        \caption{4 holes}
        \label{fig:third}
    \end{subfigure}
    \hfill
    \begin{subfigure}{0.24\textwidth}
        \centering
        \includegraphics[width=\linewidth]{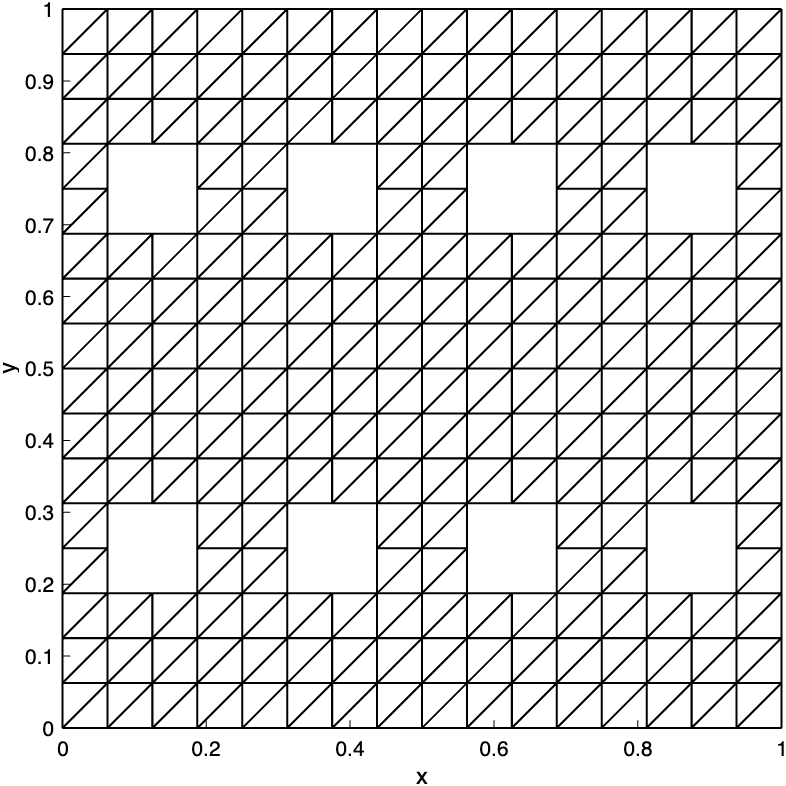}
        \caption{8 holes}
        \label{fig:fourth}
    \end{subfigure}

    \caption{Plot of coarsest mesh for domains with different numbers of holes.}
    \label{fig:four_cases}
\end{figure}

\begin{figure}[!tb]
\begin{center}
\includegraphics[width=0.45\linewidth]{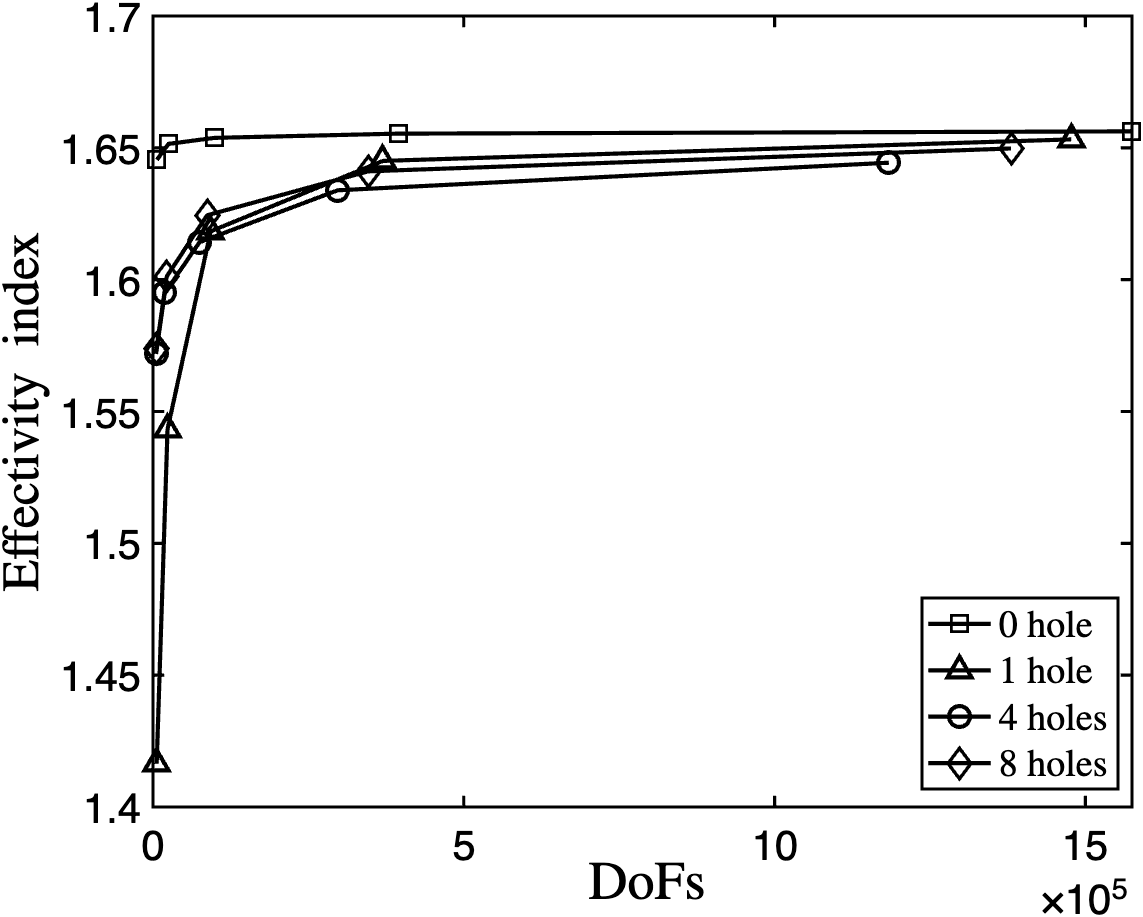}
\qquad
\includegraphics[width=0.45\linewidth]{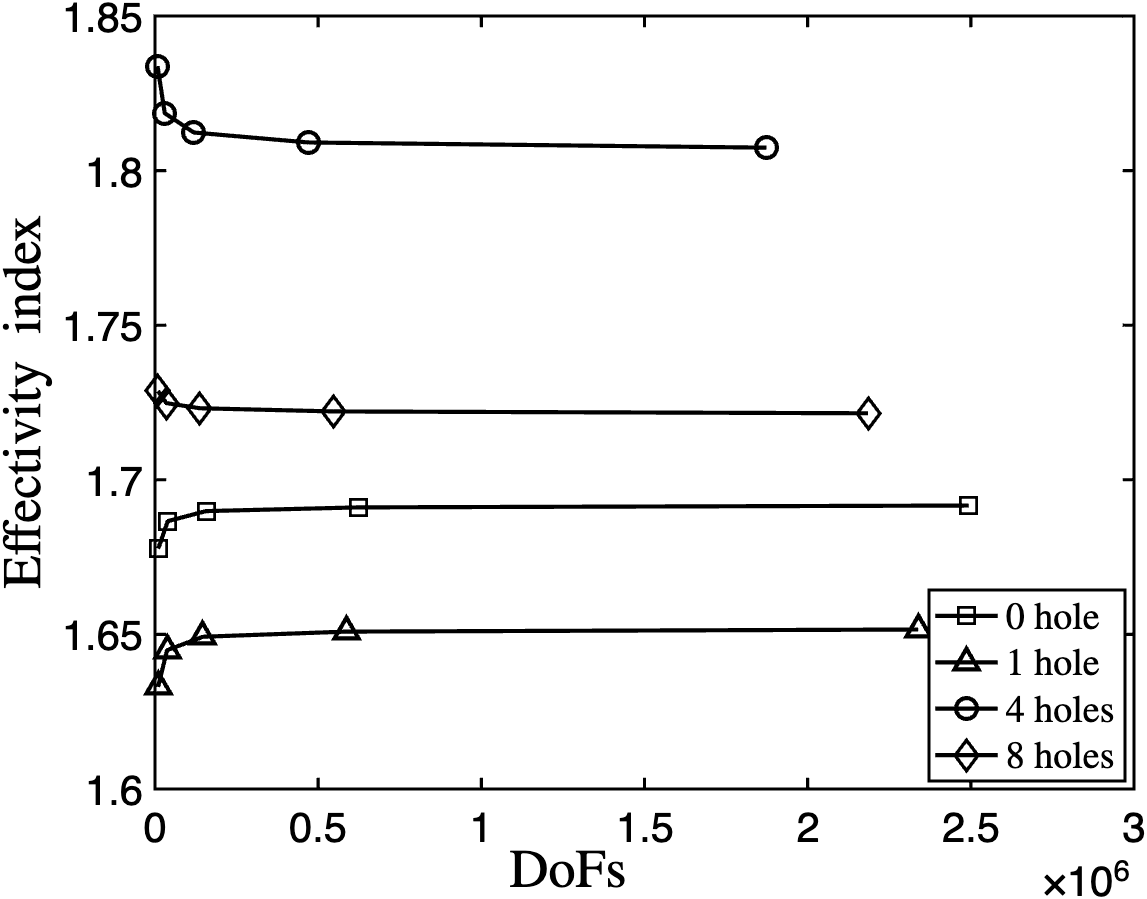}
\end{center}
\caption{Example 2: Effectivity index plotted against number of degrees of freedom for domains with $n$ holes, with $n\in\{0,1,4,8\}$. The polynomial degree is $k=0$ in the left panel and $k=1$ in the right panel.}
\label{fig:effectivity-k0-k1}
\end{figure}

\subsection{Example 3: Adaptive algorithm for 2D singular solution}
\label{sec:num_adaptive_2D}

In this example, we choose the source term $f$ and the boundary conditions on the
L-shaped domain $\Omega := (-1,1)^2 \setminus \{(0,1)\times(-1,0)\}$, so that the exact solution in polar coordinates is
\begin{equation}
u = r^{\frac43} \sin(4\theta/3).
\end{equation}
We test an $h$-adaptive algorithm driven by the a posteriori error estimator from Section~\ref{sec:apost}. The adaptive algorithm starts from a coarse mesh and uses the estimator from Theorem \ref{thm:upperbound} to mark mesh cells for refinement through a bulk-chasing criterion with parameter $30\%$ (also known as D\"orfler's marking). The set of marked elements is used to create a new, finer triangulation. Altogether, the adaptive algorithm can be classically described as a loop performing the following four tasks at each step:
$$
\text{SOLVE} \longrightarrow \text{ESTIMATE} \longrightarrow  \text{MARK} \longrightarrow \text{REFINE}.
$$
We first examine the convergence behavior of the adaptive algorithm for
$k \in \{0,1,2,3\}$. The energy error
and the a posteriori error estimator are presented
in Figure~\ref{Ex3:new_error_estimator_k=0,1,2,3}
as a function of DoFs. As there is no data oscillation term, the minimum on the right-hand side of~\eqref{upper error bound} is zero. Since the exact solution does not satisfy the homogeneous boundary conditions associated with the space $H_0^2(\Omega)$, the estimators are modified according to Remark~\ref{remark: Inhomogeneous BC} in order to incorporate the resulting inhomogeneous boundary conditions.
We observe that both
the energy error and the a posteriori estimator converge at the optimal rate
$\mathcal{O}(\mathrm{DoFs}^{-\frac{(k+1)}{2}})$. In contrast, under uniform refinement,
the energy error and the estimator converge only at the suboptimal rate
$\mathcal{O}(\mathrm{DoFs}^{-\frac16})$, independently of $k$. Moreover, we observe in the left panel of Figure~\ref{Ex3:new_effectivity k=0, 1,2,3} that the effectivity index remains well behaved as a function of DoFs, taking values between 1.5 and 1.8.
Finally, we report in the right panel of Figure~\ref{Ex3:new_effectivity k=0, 1,2,3} the effectivity index as a function of the polynomial degree $k\in\{0,\ldots,12\}$ on a mesh consisting of $96$ triangular cells. We observe an algebraic growth rate of $k^{\frac12}$, which is less than the rate $k^{\frac32}$ observed numerically for dG methods in \cite[Section 5.1]{DongMascottoSutton2021}.

\begin{figure}[!tb]
\begin{center}
\begin{tabular}{cc}
\includegraphics[width=0.45\linewidth]{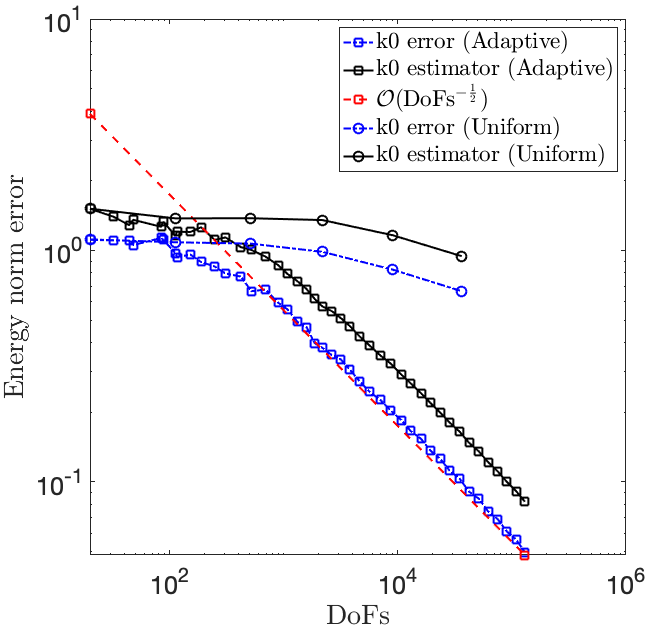} &
\includegraphics[width=0.45\linewidth]{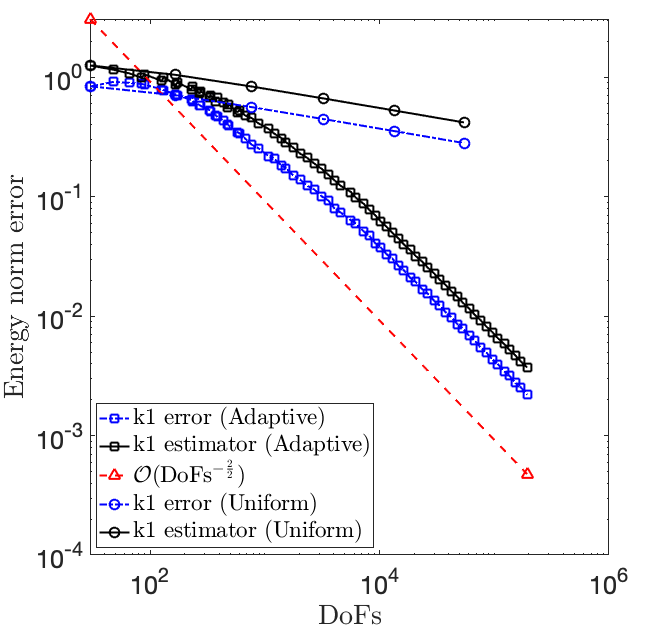} \\
\includegraphics[width=0.45\linewidth]{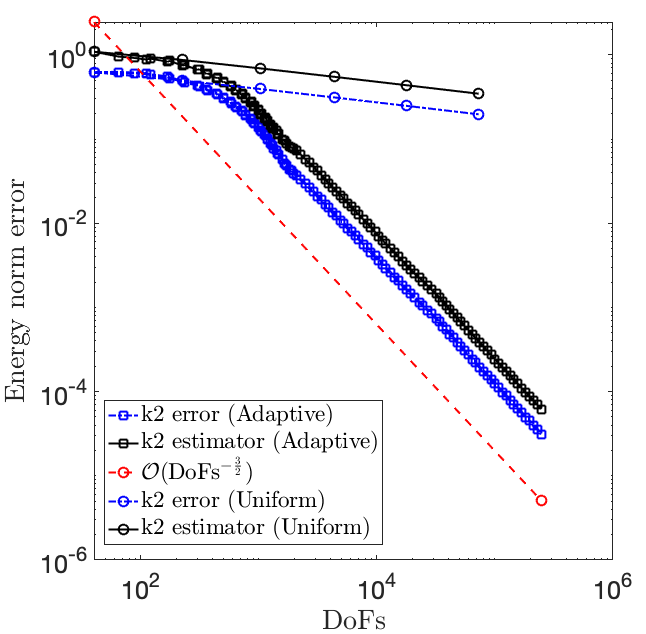} &
\includegraphics[width=0.45\linewidth]{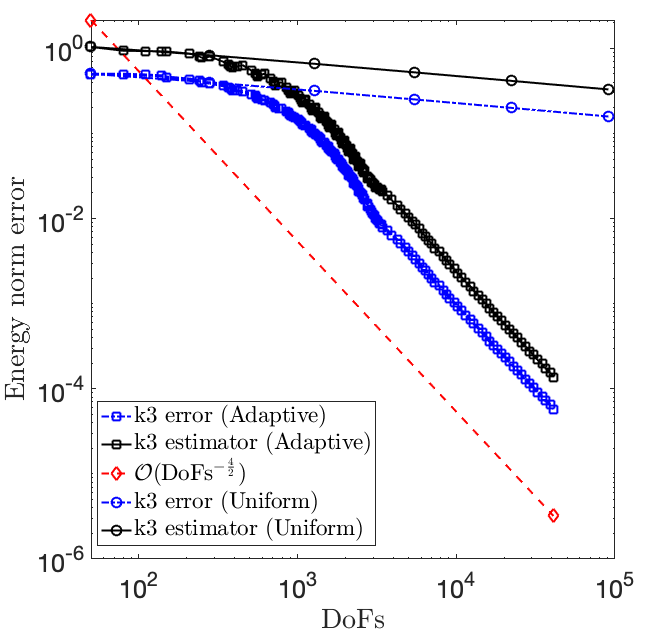}
\end{tabular}
\end{center}
\caption{Example 3: Energy error and a posteriori error estimator as a function of DoFs for $k\in\{0,1,2,3\}$.}\label{Ex3:new_error_estimator_k=0,1,2,3}
\end{figure}
\begin{figure}[!tb]
\begin{center}
\includegraphics[width=0.4\linewidth]{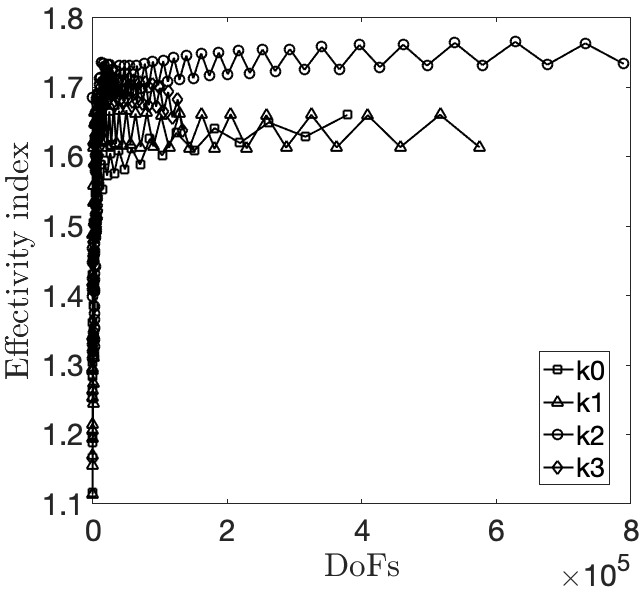}
\qquad
\includegraphics[width=0.4\linewidth]{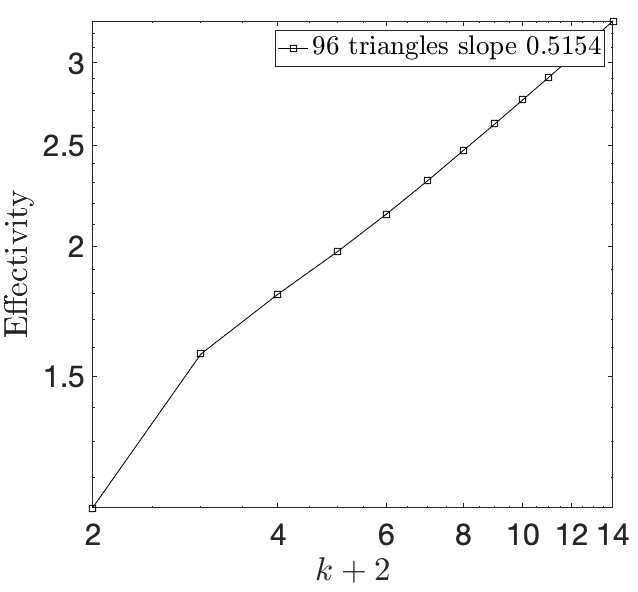}
\end{center}
\caption{Example 3: Effectivity index as a function of DoFs for $k\in\{0, 1,2,3\}$ (left panel). Effectivity index as a function of $k\in\{0, \ldots,12\}$ on a mesh composed of 96 cells (right panel).
}\label{Ex3:new_effectivity k=0, 1,2,3}
\end{figure}

\subsection{Example 4: Adaptive algorithm for 3D solution with point singularity}
\label{sec:num_adaptive_3D}

In this example, we consider the three-dimensional computational domain $\Omega:=(-1,1)^3$.
The source term $f$ and the boundary data are prescribed so that the exact solution is given by
$u := r((1-x^2)(1-y^2)(1-z^2))^2$, with $r:=(x^2+y^2+z^2)^{\frac12}$. We notice $u\in H^{2.5-\epsilon}(\Omega)$ for any arbitrarily small $\varepsilon > 0$.
We test the same $h$-adaptive algorithm as in the previous section.
We examine the convergence behavior of the adaptive algorithm for
$k \in \{0,1,2,3\}$. The energy error
and the a posteriori error estimator are presented
in Figure~\ref{Ex4:new_error_estimator_k=0,1,2,3 for 3D}
as a function of DoFs.  We observe that both
the energy error and the a posteriori estimator converge at the optimal rate
$\mathcal{O}\bigl(\mathrm{DoFs}^{-\frac{(k+1)}{3}}\bigr)$. The corresponding effectivity indices for $k\in\{0,1,2,3\}$ are displayed in Figure~\ref{Ex4:new_effectivity k=0, 1,2,3 for 3D}. We observe that the effectivity indices remain well behaved as a function of DoFs, taking values between 1.3 and 2.0 in the present three-dimensional setting.

\begin{figure}[!tb]
\begin{center}
\begin{tabular}{cc}
\includegraphics[width=0.45\linewidth]{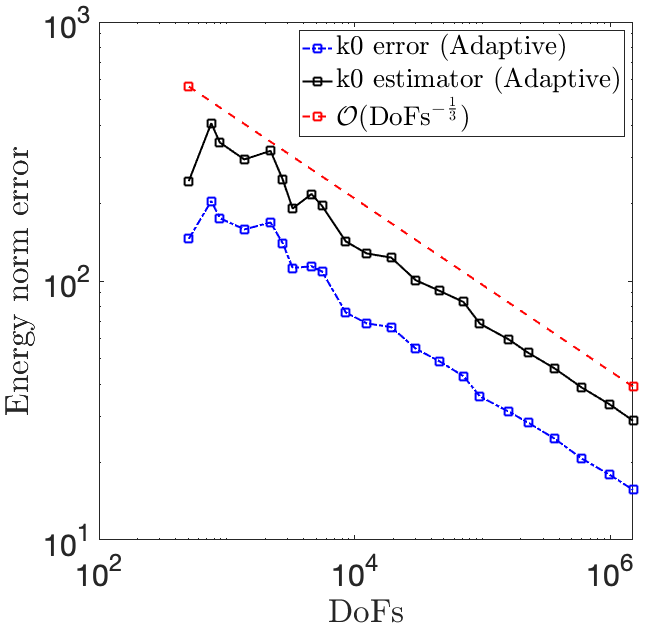} &
\includegraphics[width=0.45\linewidth]{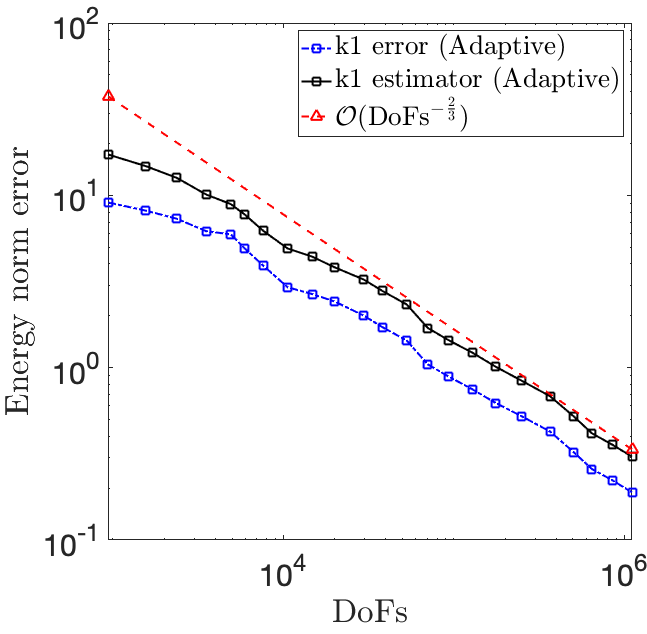} \\
\includegraphics[width=0.45\linewidth]{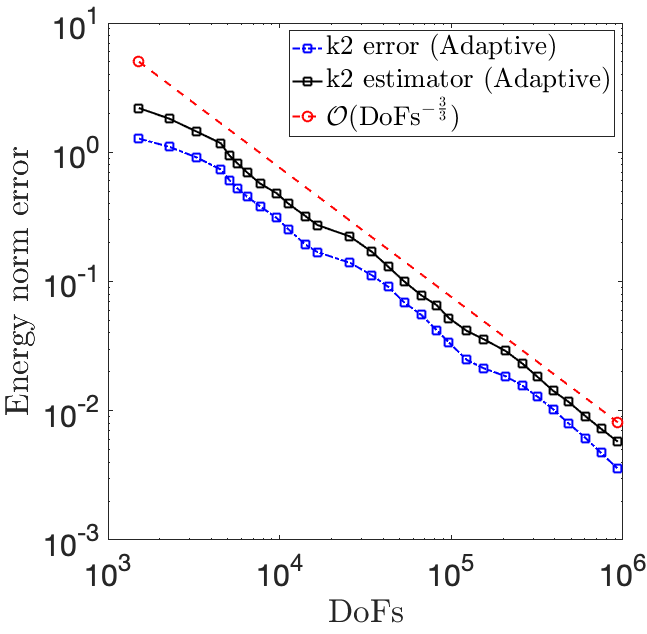} &
\includegraphics[width=0.45\linewidth]{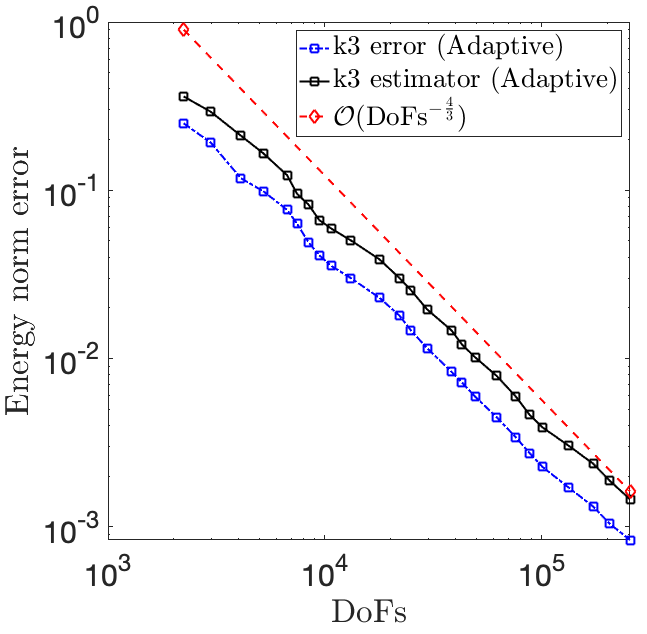}
\end{tabular}
\end{center}
\caption{Example 4: Energy error and a posteriori error estimator as a function of DoFs for $k\in\{0,1,2,3\}$.}\label{Ex4:new_error_estimator_k=0,1,2,3 for 3D}
\end{figure}
\begin{figure}[!tb]
\begin{center}
\includegraphics[width=0.4\linewidth]{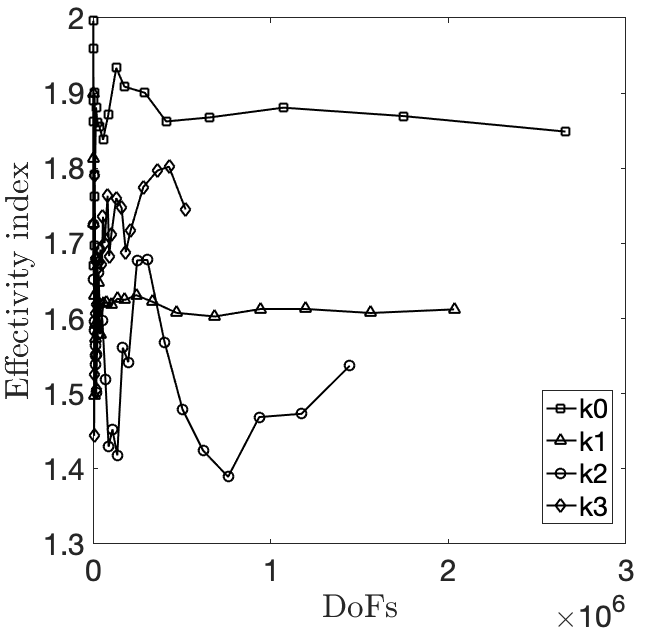}
\end{center}
\caption{Example 4: Effectivity index as a function of DoFs for $k\in\{0, 1,2,3\}$.
}\label{Ex4:new_effectivity k=0, 1,2,3 for 3D}
\end{figure}

\section{Technical proofs}\label{sec:proofs}
This section contains the proofs of Lemmas \ref{lemma: dual norm of residual} and \ref{lemma: nonconforming error total}.

\subsection{Proof of Lemma \ref{lemma: dual norm of residual}}\label{sec: proof dual norm of residual}

In this section, we prove Lemma \ref{lemma: dual norm of residual}. Specifically,
we prove in Section~\ref{sec: proof CH interpolation} that,
under Assumption~\ref{assumption 1},
\begin{subequations}
\begin{equation} \label{eq:first_bnd_res}
\|\mathcal{R}_{\mesh}\|_{H^{-2}(\Omega)}^2 \lesssim \sum_{T\in\mesh} (k+2)\big\{
\eta_{T,{\rm sta}}^2 + \mathcal{O}(f)^2_T\big\},
\end{equation}
and we prove in Section~\ref{sec: proof BS interpolation} that
\begin{equation} \label{eq:second_bnd_res}
\|\mathcal{R}_{\mesh}\|_{H^{-2}(\Omega)}^2 \lesssim \sum_{T\in\mesh} \big\{
(k+2) \eta_{T,{\rm sta}}^2 + \eta_{T,{\rm res}}^2+ \mathcal{O}(f)^2_T\big\}.
\end{equation}
\end{subequations}
The bound~\eqref{eq:second_bnd_res} is nothing but~\eqref{eq:dual_res1}, and the combination of~\eqref{eq:first_bnd_res}-\eqref{eq:second_bnd_res} readily gives~\eqref{eq:dual_res2}.

Recalling the definition~\eqref{eq:def_res_dual} of the dual residual norm $\|\mathcal{R}_{\mesh}\|_{H^{-2}(\Omega)}$, we need to bound $(\hes_{\mesh} e, \hes v)_{\Omega}$ for all $v\in H^2_0(\Omega)$.
Invoking the weak formulation \eqref{weakform} and adding/subtracting $(\hes R^{k+2}_T (\widehat{u}_T), \hes v)_{T}$, we obtain
\begin{align*}
(\hes_{\mesh} e, \hes v)_{\Omega} = &\sum_{T \in \mesh} \big\{(f,v)_{T} - (\hes R^{k+2}_T (\widehat{u}_T), \hes v)_{T} + (\hes R^{k+2}_T (\widehat{u}_T), \hes v)_{T}  \\
&\quad - (\hes u_T, \hes v)_{T} \big\}.
\end{align*}
Moreover, using the discrete HHO scheme \eqref{discrete_prob}, we infer that, for all
$\widehat{w}_h \in \widehat{V}^k_{h0}$,
\begin{align}\label{bound 41}
(\hes_{\mesh} e, \hes v)_{\Omega}
= {}&\sum_{T \in \mesh} \big\{
(f, v-w_T)_{T}
-(\hes  R^{k+2}_T (\widehat{u}_T), \hes v-\hes  R^{k+2}_T (\widehat{w}_T))_{T}
\notag \\
&
+ (\hes  (R^{k+2}_T (\widehat{u}_T) -  u_T), \hes v)_{T}
+ S_{\dK}(\widehat{u}_T, \widehat{w}_T) \big\}.
\end{align}
We need now to select a suitable discrete test function $\widehat{w}_h\in \widehat{V}^k_{h0}$.

\subsubsection{Bound using the canonical interpolation operator and Assumption~\ref{assumption 1}}\label{sec: proof CH interpolation}

The first possibility is to set $\widehat{w}_h := \widehat{\mathcal{I}}^k_h(v)$ in~\eqref{bound 41} (see~\eqref{def: global HHO interpolation CH}). Using the $H^2$-elliptic projection property \eqref{eq:H2_ell_proj}, we readily infer that the second term on the right-hand side of~\eqref{bound 41} vanishes, so that
\begin{align*}
(\hes_{\mesh} e, \hes v)_{\Omega}
&= \sum_{T \in \mesh} \Big\{
(f, v -\cC^{k+2}_T(v))_{T}  \\
&\quad
+ (\hes  (R^{k+2}_T (\widehat{u}_T) - u_T), \hes v)_{T} + S_{\dK}(\widehat{u}_T,  \hat{\mathcal{I}}_{T}^k(v))\Big\}.
\end{align*}
Using the orthogonality property \eqref{cell property} for $\cC_T^{k+2}$, the Cauchy--Schwarz inequality, and invoking the estimate from Lemma~\ref{prop}, we obtain
\begin{align*}
(\hes_{\mesh} e, \hes v)_{\Omega}
\lesssim{}& \sum_{T \in \mesh} \Big\{
(f  - \Pi^{k-2}_T (f), v - \cC^{k+2}_T (v))_{T} \\
&+ S_{\dK}(\widehat{u}_T, \widehat{u}_T)^{\frac{1}{2}}
\big\{\|\hes v\|_{T}
+ S_{\dK}( \hat{\mathcal{I}}_{T}^k(v), \hat{\mathcal{I}}_{T}^k(v))^{\frac{1}{2}} \big\} \Big\}.
\end{align*}
Using the bound \eqref{Stabilisation of reduction bound: CH} from Lemma~\ref{lem:bound_stab}, we obtain
\begin{align}\label{bound 45}
(\hes_{\mesh} e, \hes v)_{\Omega} \lesssim {}& \sum_{T \in \mesh} \Big\{ \| f  - \Pi^{k-2}_T (f) \|_{T}\|v - \cC^{k+2}_T (v)\|_{T}\nonumber\\
&\qquad
+ S_{\dK}(\widehat{u}_T, \widehat{u}_T)^{\frac{1}{2}} (k+2)^{\frac{1}{2}}\|\hes v\|_{T}
\Big\}.
\end{align}
Next, we bound $\|v - \cC^{k+2}_T (v)\|_{T}$. For all $k\in\{0,\ldots,d-1=2\}$, the vertex prescription condition~\eqref{vertex property} implies that $\mathcal{L}^1_T (v - \cC^{k+2}_T v)=0$, with the linear Lagrange interpolation operator $\mathcal{L}^1_T$ in the mesh cell $T$. Invoking the $H^2$-stability of $ \cC^{k+2}_T$,  we infer that
\begin{equation*}
\|v - \cC^{k+2}_T (v)\|_{T}= \|(I - \mathcal{L}^1_T )(v - \cC^{k+2}_T (v))\|_{T}\lesssim h_T^2 \|\hes (v - \cC^{k+2}_T (v))\|_{T} \lesssim h_T^2 \|\hes v \|_{T},
\end{equation*}
where $I$ is the identity operator.
In the case $k\geq d=3$, using the orthogonality property \eqref{cell property}, we obtain
\begin{align*}
\|v - \cC^{k+2}_T (v)\|_{T}^2
&= (v - \cC^{k+2}_T (v), v - \cC^{k+2}_T (v))_{T} \\
&=(v - \cC^{k+2}_T (v), (I-\mBST^{k-2})(v - \cC^{k+2}_T (v)))_{T},
\end{align*}
so that
$$
\|v - \cC^{k+2}_T (v)\|_{T} \le \|(I-\mBST^{k-2})(v - \cC^{k+2}_T (v))\|_{T}.
$$
Invoking the approximation property \eqref{eq: Modified BS approximation} gives
\begin{equation*}
\|(I-\mBST^{k-2})(v - \cC^{k+2}_T (v))\|_{T}
\lesssim \bigg\{\frac{h_T}{k-2}\bigg\}^2
\|\hes(v - \cC^{k+2}_T (v))\|_{T}.
\end{equation*}
We further apply the triangle inequality to obtain
$$
\|\hes(v - \cC^{k+2}_T (v))\|_{T} \leq  \|\hes(v -  R_T^{k+2}(\widehat{\mathcal{I}}_{T}^k(v)))\|_{T} + \|\hes( R_T^{k+2}(\widehat{\mathcal{I}}_{T}^k(v)) - \cC^{k+2}_T (v))\|_{T}.
$$
Owing to the $H^2$-elliptic projection property \eqref{eq:H2_ell_proj}, the first term on the right-hand side is bounded by $\|\hes v\|_{T}$. Moreover, owing to the bound~\eqref{HHO relation} from Lemma~\ref{prop} and the bound \eqref{Stabilisation of reduction bound: CH} from Lemma~\ref{lem:bound_stab}, we obtain
$$
\|\hes( R_T^{k+2}(\widehat{\mathcal{I}}_{T}^k(v)) - \cC^{k+2}_T (v))\|_{T} \lesssim
S_{\dK}(\widehat{\mathcal{I}}_{T}^k(v), \widehat{\mathcal{I}}_{T}^k(v))^{\frac12} \lesssim
(k+2)^{\frac12}\|\hes v\|_{T}.
$$
As a result, we have $\|\hes(v - \cC^{k+2}_T (v))\|_{T} \lesssim (k+2)^{\frac{1}{2}}\|\hes v\|_{T}$. Putting the above bounds together, we infer that, for all $k\ge d=3$,
$$
\|v - \cC^{k+2}_T (v)\|_{T} \lesssim \bigg\{\frac{h_T}{k-2}\bigg\}^2
(k+2)^{\frac{1}{2}}\|\hes v\|_{T}.
$$
Combining this estimate with the bounds on $\|v - \cC^{k+2}_T (v)\|_{T}$ for all $k\in\{0,\ldots,d-1=2\}$, we infer that
$$
\|v - \cC^{k+2}_T (v)\|_{T} \lesssim \hbar_T^2 (k+2)^{\frac{1}{2}}\|\hes v\|_{T},
$$
recalling that $\hbar_T:= \frac{h_T}{k+2}$ and $k+2\lesssim k-2$ for $k\ge3$.
Inserting this bound in~\eqref{bound 45} gives
$$
(\hes_{\mesh} e, \hes v)_{\Omega} \lesssim \sum_{T \in \mesh} \Big\{ \hbar_T^2\| f  - \Pi^{k-2}_T (f) \|_{T}
+ S_{\dK}(\widehat{u}_T, \widehat{u}_T)^{\frac{1}{2}} \Big\} (k+2)^{\frac{1}{2}}\|\hes v\|_{T}.
$$
Recalling the definitions \eqref{eq:def_eta_sta} and \eqref{est:oscillation} of $\eta_{T,{\rm sta}}^2$ and $\mathcal{O}(f)^2_T$, respectively, we conclude that
$$
(\hes_{\mesh} e, \hes v)_{\Omega} \lesssim \bigg\{ \sum_{T \in \mesh} (k+2)\big\{ \eta_{T,{\rm sta}}^2 + \mathcal{O}(f)^2_T\big\} \bigg\}^{\frac12} \|\hes v\|_\Omega.
$$
This proves that, under Assumption~\ref{assumption 1}, \eqref{eq:first_bnd_res} holds true.

\subsubsection{Bound using the Babu\v{s}ka--Suri interpolation operator} \label{sec: proof BS interpolation}

We now set $\widehat{w}_h := \widehat{\mathcal{J}}^k_h(v)$ in~\eqref{bound 41} (see~\eqref{def: global HHO interpolation BS}).
Let us set $\zeta:=v-\mBS^{k+2}(v)$.
Invoking the definition~\eqref{reconstruct_a} of the reconstruction operator in the second term on the right-hand side of~\eqref{bound 41} (here, $R^{k+2}_T (\widehat{u}_T)$ plays the role of the discrete test function), we obtain
\begin{align*}
(\hes  R^{k+2}_T (\widehat{u}_T), \hes v-\hes  R^{k+2}_T (\widehat{w}_T))_{T}
= {}&(\hes  R^{k+2}_T (\widehat{u}_T), \hes \zeta)_{T} \\
&+ \big(\partial_{nn}  R^{k+2}_T (\widehat{u}_T), \partial_n \mBS^{k+2}(v)- \Pi^k_{\dK}(\partial_{n} v)\big)_{\dK},
\end{align*}
where we used that the second component of $\widehat{\mathcal{J}}^k_h(v)$ is the trace on the mesh skeleton of its first component. Substituting this expression in~\eqref{bound 41} gives
\begin{align*}
(\hes_{\mesh} e, \hes v)_{\Omega}
&= \sum_{T \in \mesh} \Big\{
(f, \zeta)_{T}
+ S_{\dK}(\widehat{u}_T, \hat{\mathcal{J}}_{T}^k(v) )
+ (\hes  (R^{k+2}_T (\widehat{u}_T) - u_T), \hes v)_{T}
\notag \\
& \quad - (\hes  R^{k+2}_T (\widehat{u}_T), \hes \zeta)_{T}
- \big(\partial_{nn}  R^{k+2}_T (\widehat{u}_T), \partial_n \mBS^{k+2}(v)- \Pi^k_{\dK}(\partial_{n} v)\big)_{\dK}
\Big\}.
\end{align*}
Applying the integration by parts formula~\eqref{eq:IPP_bis} to the fourth term on the right-hand side, we infer that
\begin{align*}
(\hes_{\mesh} e, \hes v)_{\Omega} ={}& \sum_{T \in \mesh} \Big\{ (f - \Delta^2  R^{k+2}_T (\widehat{u}_T) , \zeta)_{T}
+ S_{\dK}(\widehat{u}_T, \hat{\mathcal{J}}_{T}^k(v) )\\
&
+ (\hes  (R^{k+2}_T (\widehat{u}_T) - u_T), \hes v)_{T}
+ (\partial_n \Delta  R^{k+2}_T (\widehat{u}_T), \zeta)_{\dKi} \\
&
 - \big(\partial_{nt}  R^{k+2}_T (\widehat{u}_T), \partial_t \zeta)_{\dKi}
- \big(\partial_{nn}  R^{k+2}_T (\widehat{u}_T), \partial_nv-\Pi_{\dK}^k(\partial_nv))_{\dKi} \Big\}.
\end{align*}
Notice that the last three terms on the right-hand side are restricted to $\dKi$ since $\zeta$ and $\partial_nv$ vanish on $\partial\Omega$. Moreover, since $\partial_{nn}  R^{k+2}_T (\widehat{u}_T) \in \mathcal{P}^k(\FK)$ for all $T \in \mesh$, the last term on the above right-hand side vanishes. In addition, since $\zeta$ and $\partial_t\zeta$ are single-valued on each interface $F \in \Fint$, we obtain
\begin{align*}
(\hes_{\mesh} e, \hes v)_{\Omega} ={}& \sum_{T \in \mesh} \Big\{ (f - \Delta^2  R^{k+2}_T (\widehat{u}_T) , \zeta)_{T}
+ S_{\dK}(\widehat{u}_T, \hat{\mathcal{J}}_{T}^k(v) )\\
& + (\hes  (R^{k+2}_T (\widehat{u}_T) -  u_T), \hes v)_{T}
+ \frac{1}{2}(\sjump{\partial_n \Delta  R^{k+2}_T (\widehat{u}_T)}, \zeta)_{\dKi}
\\
&  - \frac12 \big(\sjump{\partial_{nt}  R^{k+2}_T (\widehat{u}_T)}, \partial_t \zeta)_{\dKi}\Big\}.
\end{align*}
Invoking the Cauchy--Schwarz inequality, we infer that
\begin{align*}
&|(\hes_{\mesh} e, \hes v)_{\Omega}| \leq
\sum_{T \in \mesh} \bigg\{ \hbar_T^2\|f - \Delta^2  R^{k+2}_T (\widehat{u}_T)\|_{T} \hbar_T^{-2}\|\zeta\|_{T} \\
&+ \|\hes  R^{k+2}_T (\widehat{u}_T) - \hes u_T\|_{T}\|\hes v\|_{T} + S_{\dK}(\widehat{u}_T, \widehat{u}_T)^{\frac{1}{2}} S_{\dK}(\hat{\mathcal{J}}_{T}^k(v), \hat{\mathcal{J}}_{T}^k(v))^{\frac{1}{2}} \\
&+  \frac12 \hbar_T^{\frac{3}{2}}\|\sjump{\partial_n \Delta  R^{k+2}_T (\widehat{u}_T)}\|_{\dKi}\hbar_T^{-\frac{3}{2}}\|\zeta\|_{\dKi} +  \frac{1}{2}\hbar_T^{\frac{1}{2}}\|\sjump{\partial_{nt}  R^{k+2}_T (\widehat{u}_T)}\|_{\dKi}\hbar_T^{-\frac{1}{2}}\|\partial_t \zeta\|_{\dKi} \bigg\}.
\end{align*}
Invoking the approximation properties of the modified Babuška--Suri operator (see Corollary \ref{cor: global interpolation of BS}), Lemma~\ref{prop} to bound $\|\hes  R^{k+2}_T (\widehat{u}_T) - \hes u_T\|_{T}$, and the bound \eqref{Stabilisation of reduction bound: BS} on $S_{\dK}(\hat{\mathcal{J}}_{T}^k(v), \hat{\mathcal{J}}_{T}^k(v))$ gives
\begin{align*}
|(\hes_{\mesh} e, \hes v)_{\Omega}| \lesssim {}&
\bigg\{ \sum_{T \in \mesh} \Big\{ \hbar_T^4\|f - \Delta^2  R^{k+2}_T (\widehat{u}_T)\|_{T}^2 + S_{\dK}(\widehat{u}_T, \widehat{u}_T) \\
&+ \hbar_T^3\|\sjump{\partial_n \Delta  R^{k+2}_T (\widehat{u}_T)}\|_{\dKi}^2
+ \hbar_T \|\sjump{\partial_{nt}  R^{k+2}_T (\widehat{u}_T)}\|_{\dKi}^2 \Big\} \bigg\}^{\frac12} \|\hes v\|_\Omega.
\end{align*}
Using the triangle inequality for the first term on the right-hand side and recalling the definitions of $\eta_{T,{\rm sta}}$, $\eta_{T,{\rm res}}$, and $\mathcal{O}(f)_T$
proves~\eqref{eq:second_bnd_res}.

\subsection{Proof of Lemma \ref{lemma: nonconforming error total}}\label{sec: proof full bound}

In this section, we prove Lemma \ref{lemma: nonconforming error total}, namely
$$
\|\hes_{\mesh} (u_c - u_{\mesh})\|^2_{\Omega} \lesssim \sum_{T \in \mesh} \big\{\eta_{T,{\rm sta}}^2
+ \eta_{T, {\rm tan}}^2 \big\}.
$$
Recalling the definition~\eqref{uc:partition} of $u_c$ and invoking the partition-of-unity property~\eqref{eq:pou}, we observe that
$$
u_c - u_{\mesh} = \sum_{\b{a}\in\vertice} \big\{ u_c^{\b{a}} - \psi_{\b{a}} u_{\mesh} \big\}.
$$
Setting $\delta_{\ba}:= u^{\ba}_c - \psi_{\ba}  u_{\mesh}$ for all $\ba \in \vertice$, and invoking the shape-regularity of the mesh, it is sufficient to prove that
\begin{equation}\label{bound enc_local}
\|\hes_{\mesh}\delta_{\ba}\|^2_{\oma} \lesssim \sum_{T \in \T_{\ba}} S_{{\dK}}(\hat{u}_T, \hat{u}_T) + \sum_{F \in \Fa} \hbar_F \big\{\|\sjump{\partial_{tt} u_{\mesh}}\|_{F}^2 + \|\sjump{\partial_{nt} u_{\mesh}}\|_{F}^2\big\},
\end{equation}
where $\hbar_F:=\frac{h_F}{k+2}$ and  $\Fa$ denotes the set of mesh faces in $\overline{\oma}$ that contain the vertex $\ba$. We also denote by $\bFa$ the collection of the mesh faces lying on $\partial\oma$ that do not contain $\ba$. Thus, the two sets $\Fa$ and $\bFa$ are disjoint, and their union contains all the mesh faces in $\overline{\oma}$.

Let us now prove~\eqref{bound enc_local}. To this purpose, we are going to invoke the following result: For all $\ba \in \vertice$, there is
a local interpolation operator $\mathcal{I}^{k+1,\ba}_{\rm{mKM}}: H^1(\oma) \to \mathbb{P}^{k+1}(\oma) \cap H^1(\oma)$
such that, for all $v \in H^1(\oma)$ and all $T \in \cT_{\ba}$, the following holds:
\begin{equation}\label{eq: Modified KM approximation}
\hbar_T^{-1}\|{v} - \mathcal{I}^{k+1,\ba}_{\rm{mKM}} (v)\|_{T} + \hbar_T^{-\frac12}\|{v} - \mathcal{I}^{k+1,\ba}_{\rm{mKM}}(v)\|_{\dK}  + \|\nabla \mathcal{I}^{k+1,\ba}_{\rm{mKM}}(v)\|_{T}
 \lesssim \| \nabla v\|_{\oma}.
\end{equation}
The original proof can be found in \cite{KarKulikMelenk15, melenk_rough}, and a modified version can be found in \cite{DongErn2025hp}, employing the $H^1$-seminorm on the right-hand side, as in the proof of Lemma~\ref{cor: global interpolation of BS}.
We decompose the proof of~\eqref{bound enc_local} into several steps.

(1) Invoking the local Helmholtz decomposition from Lemma~\ref{lemma:Helmholtz-dec} with ${\ut{\b{\Sigma}}} := \hes_{\mesh} \delta_{\ba}$, we obtain
\begin{align*}
\|\hes_{\mesh} \delta_{\ba}\|_{\oma}^2
= \big(\hes_{\mesh} \delta_{\ba},\, \hes {\xi}\big)_{\oma}
+ \big(\hes_{\mesh} \delta_{\ba},\, \curlrw{\ut{\bpsi}}\big)_{\oma}
+ \big(\hes_{\mesh} \delta_{\ba},\, \sk(\b{\rho})\big)_{\oma}.
\end{align*}
The first term vanishes by virtue of \eqref{def: vertex reconstruction} since
$\xi\in H^2_0(\oma)$,
whereas the last term vanishes owing to the skew-symmetry of $\sk(\b{\rho})$.
Next, adding/subtracting the term  $(\hes_{\mesh} \delta_{\ba},\, \curlrw {\mKM(\b{\ut{\psi}} )})_{\oma}$ gives
\begin{align*}
\|\hes_{\mesh} \delta_{\ba}\|_{\oma}^2
&= \big(\hes_{\mesh} \delta_{\ba},\curlrw{(\ut{\b{\psi}} - \mKM (\b{\ut{\psi}}) )}\big)_{\oma} + \big(\hes_{\mesh} \delta_{\ba},\curlrw{\mKM(\b{\ut{\psi}})}\big)_{\oma} \\
&=: T_1 +T_2.
\end{align*}
The terms $T_1$ and $T_2$ are estimated independently.

(2) Bound on $T_1$. Introducing the column vectors $\{\b{\psi}_j\}_{j\in\{1{:}d\}}$ obtained by transposing the rows of $\b{\ut{\psi}}$ and recalling the definition of $\delta_{\ba}$, we have
$$
T_1= \sumj \big(\nabla_{\mesh} \partial_j(u_c^{\ba}-\psi_{\ba}u_{\mesh}),\nabla {\times} (\b{\psi}_j - \mKM (\b{\psi}_j))\big)_{\oma},
$$
where we used the summation convention on repeated indices and where $\partial_j$ acts elementwise.
We integrate by parts the curl operator in each tetrahedron $T\in \mesha$.
Since $u_c^{\ba} \in H^2_0(\oma)$, $\b{n}_{\oma}{\times} \nabla\partial_j u_c^{\ba} = \mbf{0}$, for all $j\in\{1{:}d\}$, on all the mesh faces composing $\partial\oma$, we have
$$
\big(\nabla_{\mesh} \partial_ju_c^{\ba},\curl(\b{\psi}_j - \mKM (\b{\psi}_j))\big)_{\oma}= 0.
$$
Moreover, since $\b{\psi}_j - \mKM(\b{\psi}_j)\in \b{H}^1(\oma)$ is single-valued on all faces $F \in  \cF_a$ and $\psi_{\ba} u_{\mesh}|_F = 0$ for all faces $F\in\bFa$, we obtain
\begin{align}\label{eqn1_nonconf}
T_1
&=  \sum\limits_{F \in \Fa} \sumj (\b{n}_{F} {\times} \sjump{\nabla \pd{j} (\psi_{\ba} u_{\mesh})} ,\b{\psi}_j - \mKM(\b{\psi}_j))_{F}.
\end{align}
Using the Cauchy--Schwarz inequality, we obtain
\begin{align*}
|T_{1}| \leq{}& \bigg\{\sum\limits_{F \in {\cF}_a} \sumj \hbar_F \| \b{n}_{F} {\times} \sjump{\nabla \pd{j} (\psi_{\ba} u_{\mesh})}  \|^2_F\bigg\}^{\frac{1}{2}}\\
&\times \bigg\{ \sum\limits_{F \in {\cF}_a} \sumj \hbar_F^{-1} \|\b{\psi}_j - \mKM(\b{\psi}_j)\|^2_F\bigg\}^{\frac{1}{2}}.
\end{align*}
Invoking the shape-regularity of the mesh, the approximation estimate~\eqref{eq: Modified KM approximation}, and the stability of the local Helmholtz decomposition (see~\eqref{bound:Helmholtz}), we infer that
$$
\bigg\{ \sum\limits_{F \in {\cF}_a} \sumj \hbar_F^{-1} \|\b{\psi}_j - \mKM(\b{\psi}_j)\|^2_F\bigg\}^{\frac{1}{2}} \lesssim \sumj \|\nabla \b{\psi}_j\|^2_{\oma} \lesssim \|\hes_{\mesh} \delta_{\ba}\|_{\oma}.
$$
Hence,
$$
|T_{1}| \lesssim \bigg\{\sum\limits_{F \in \Fa} \sumj \hbar_F \| \b{n}_F {\times} \sjump{\nabla \pd{j} (\psi_{\ba} u_{\mesh})}\|^2_F\bigg\}^{\frac{1}{2}}
\|\hes_{\mesh} \delta_{\ba}\|_{\oma}.
$$
Next, applying the product rule gives
\begin{equation*}
\nabla(\pd{j} (\psi_{\ba} u_{\mesh}))
= \psi_{\ba}\nabla \pd{j} u_{\mesh}
+ u_{\mesh}\nabla \pd{j} \psi_{\ba}
+ \pd{j} \psi_{\ba}\nabla u_{\mesh}
+ \pd{j} u_{\mesh}\nabla \psi_{\ba}.
\end{equation*}
Since $\psi_{\ba}$ is single-valued at $F$ and $\|\psi_{\ba}\|_{L^\infty(F)} \lesssim 1$, we infer that
\begin{align*}
\sum\limits_{F \in \Fa} \sumj \hbar_F \| \b{n}_F {\times} \sjump{\psi_{\ba}\nabla \pd{j} u_{\mesh}}\|^2_F &=\sum\limits_{F \in \Fa} \sumj \hbar_F \| \psi_{\ba} \b{n}_F {\times} \sjump{\nabla \pd{j} u_{\mesh}}\|^2_F \\
&\lesssim \sum\limits_{F \in \Fa} \sumj \hbar_F \| \b{n}_F {\times} \sjump{\nabla \pd{j} u_{\mesh}}\|^2_F \\
&\lesssim \sum\limits_{F \in \Fa} \hbar_F \big\{\|\sjump{\partial_{tt} u_{\mesh}}\|_{F}^2 + \|\sjump{\partial_{nt} u_{\mesh}}\|_{F}^2\big\}.
\end{align*}
Since $\n_F \times \nabla \pd{j}\psi_{\ba}$ is single-valued at $F$, $\|\nabla \pd{j}\psi_{\ba}\|_{\bL^{\infty}(F)} \lesssim h_F^{-2}$, and $h_F^{-1}\le \hbar_F^{-1}$, we infer that
\begin{align*}
\sum\limits_{F \in \Fa} \sumj \hbar_F \| \n_F \times \sjump{u_{\mesh}\nabla \pd{j} \psi_{\ba}} \|^2_F &\lesssim \sum\limits_{F \in \Fa} h_F^{-4} \hbar_F \|\sjump{u_{\mesh}}\|_F^2
\le \sum\limits_{F \in \Fa} \hbar_F^{-3} \|\sjump{u_{\mesh}}\|_F^2.
\end{align*}
Since $\pd{j}\psi_{\ba}$ is single-valued at $F$ and $\|\pd{j}\psi_{\ba}\|_{L^{\infty}(F)} \lesssim h_F^{-1}$, we infer that
\begin{align*}
&\sum\limits_{F \in \Fa} \sumj \hbar_F \| \n_F \times \sjump{\pd{j} \psi_{\ba}\nabla u_{\mesh}} \|^2_F\\ &\lesssim \sum\limits_{F \in \Fa} h_F^{-2}\hbar_F\| \n_F \times \sjump{\nabla u_{\mesh}} \|^2_F
\lesssim \sum\limits_{F \in \Fa} (k+2)^2 h_F^{-2}\hbar_F^{-1}\| \sjump{u_{\mesh}} \|^2_F
= \sum\limits_{F \in \Fa} \hbar_F^{-3}\| \sjump{u_{\mesh}} \|^2_F,
\end{align*}
where the first bound on the second line follows from the inverse inequality \eqref{eq:discrete inverse} on $F$.
Finally, since $\nabla\psi_{\ba}$ is single-valued at $F$ and $\|\nabla\psi_{\ba}\|_{\bL^\infty(F)} \lesssim h_F^{-1}$, we infer that
\begin{align*}
\sum\limits_{F \in \Fa} \sumj \hbar_F \| \n_F \times \sjump{\pd{j} u_{\mesh}\nabla \psi_{\ba}} \|^2_F &\lesssim \sum\limits_{F \in \Fa} \sumj h_F^{-2} \hbar_F \| \sjump{\partial_ju_{\mesh}}\|_F^2.
\end{align*}
Invoking the triangle inequality gives
\begin{align*}
\| \sjump{\partial_ju_{\mesh}}\|_F &\le \| (I - \Pi^k_F) \sjump{\pd{j} u_{\mesh}}\|_F
+\|\Pi^k_F \sjump{\pd{j} u_{\mesh}}\|_F \\
&\lesssim \hbar_F \| \sjump{\partial_t\partial_j u_{\mesh}} \|_F + \|\Pi^k_F \sjump{\partial_n u_{\mesh}}\|_F + \|\Pi^k_F \sjump{\partial_t u_{\mesh}}\|_F \\
&\lesssim \hbar_F \big\{ \|\sjump{\partial_{tt} u_{\mesh}}\|_{F} + \|\sjump{\partial_{nt} u_{\mesh}}\|_{F} \big\}\\
&\quad + \|\Pi^k_F \sjump{\partial_n u_{\mesh}}\|_F + (k+2)\hbar_F^{-1} \|\sjump{u_{\mesh}}\|_F,
\end{align*}
where the bounds on the second line result from the $hp$-approximation properties of $\Pi^k_F$ on $F$ (for the first term) and the decomposition of any directional derivative into normal and tangential derivatives (for the second term), and the bounds on the third line result from the same decomposition (for the first term) and the $L^2$-stability of $\Pi^k_F$ and
the inverse inequality \eqref{eq:discrete inverse} on $F$ (for the third term).
This implies that
\begin{align*}
\sum\limits_{F \in \Fa} \sumj \hbar_F \| \n_F \times \sjump{\pd{j} u_{\mesh} &\nabla \psi_{\ba}} \|^2_F \lesssim {} \sum\limits_{F \in \Fa} \Big\{ \hbar_F \big\{\|\sjump{\partial_{tt} u_{\mesh}}\|_{F}^2 + \|\sjump{\partial_{nt} u_{\mesh}}\|_{F}^2\big\} \\
& + (k+2) \hbar_F^{-1} \|\Pi^k_F \sjump{\partial_n u_{\mesh}}\|_F^2
+(k+2)  \hbar_F^{-3} \|\sjump{u_{\mesh}}\|_F^2 \Big\}.
\end{align*}
Putting the above bounds together, we infer that
\begin{multline}\label{eq: IPDG T1}
|T_1| \lesssim \bigg\{ \sum\limits_{F \in \Fa} \Big\{ \hbar_F \big\{\|\sjump{\partial_{tt} u_{\mesh}}\|_{F}^2 + \|\sjump{\partial_{nt} u_{\mesh}}\|_{F}^2\big\} \\
+ (k+2)\hbar_F^{-1} \|\Pi^k_F \sjump{\partial_n u_{\mesh}}\|_F^2
+ (k+2)\hbar_F^{-3} \|\sjump{u_{\mesh}}\|_F^2 \Big\}
\bigg\}^{\frac12} \|\hes_{\mesh}\delta_{\ba}\|_{\oma}.
\end{multline}
Adding and subtracting $\gamma_F$ and $u_F$ to the third and fourth terms on the right-hand side, invoking the triangle inequality, recalling the definition~\eqref{def:stab} of the stabilization bilinear form, we obtain
\begin{equation} \label{T1:BOUND}
|T_1| \lesssim \bigg\{ \sum\limits_{F \in \Fa} \Big\{ \hbar_F \big\{\|\sjump{\partial_{tt} u_{\mesh}}\|_{F}^2 + \|\sjump{\partial_{nt} u_{\mesh}}\|_{F}^2\big\}
+ \sum_{T \in \T_{\ba}} S_{{\dK}}(\hat{u}_T, \hat{u}_T)
\bigg\}^{\frac12} \|\hes_{\mesh}\delta_{\ba}\|_{\oma}.
\end{equation}

(3) Bound on $T_2$. Since $\mKM(\ut{\b{\psi}}) \in \ut{\b{H}}^1(\oma)$, integrating by parts the gradient operator gives
\begin{align}\label{eq: IPDG T2}
T_2 &= \big(\hes u^{\ba}_c,\, \curlrw \mKM(\b{\ut{\psi}})\big)_{\oma} - \big(\hes_{\mesh} (\psi_{\ba} u_{\mesh}),\, \curlrw \mKM(\b{\ut{\psi}})\big)_{\oma} \nonumber \\
&= -\sum_{F \in \Fa} \big( \sjump{\nabla_{\mesh} (\psi_{\ba} u_{\mesh})}, \curlrw{\mKM(\b{\ut{\psi}})}\b{n}_F\big)_F,
\end{align}
recalling that $\psi_{\ba}$ vanishes on all faces $F\in\bFa$ and observing that $\curlrw{\mKM(\b{\ut{\psi}})}\b{n}_F$ is single-valued on those faces.
Next, using the product rule $\nabla_{\mesh}(\psi_{\ba} u_{\mesh})= u_{\mesh} \nabla \psi_{\ba} +\psi_{\ba}\nabla_{\mesh} u_{\mesh}$, we obtain
\begin{align*}
-T_2  ={}& \sum_{F \in \Fa} \Big\{  \big(\nabla \psi_{\ba} \sjump{u_{\mesh}},\curlrw{\mKM(\b{\ut{\psi}})}\b{n}_F\big)_F\\
& +\big( (I - \Pi^k_F) \sjump{\nabla_{\mesh} u_{\mesh}}, \psi_{\ba} \curlrw {\mKM(\b{\ut{\psi}})}\b{n}_F \big)_F \\
& +\big(\Pi^k_F\sjump{\nabla_{\mesh}u_{\mesh}},\psi_{\ba} \curlrw{\mKM(\b{\ut{\psi}})}\b{n}_F\big)_F \Big\},
\end{align*}
and we bound the three terms on the right-hand side, say $T_{21}$, $T_{22}$, and $T_{23}$.
Since $\|\nabla \psi_{\ba}\|_{\bL^{\infty}(F)} \lesssim  h_F^{-1}$, invoking the Cauchy--Schwarz inequality and recalling the definition~\eqref{def:stab} of the stabilization bilinear form gives
\begin{align*}
|T_{21}|& \lesssim
\sum_{F \in \Fa} h^{-1}_F\|\sjump{u_{\mesh}}\|_F\|\curlrw{\mKM(\b{\ut{\psi}})}\b{n}_F\|_F \\
&\lesssim \bigg\{\sum_{F \in \cF_a} \hbar_F^{-3}\frac{(k+2)^2}{h^3_F}\|\sjump{u_{\mesh}}\|^2_F\bigg\}^{\frac{1}{2}}\bigg\{\sum_{F \in \cF_a}\frac{h_F}{(k+2)^2}\|\curlrw{\mKM(\b{\ut{\psi}})}\b{n}_F\|^2_F\bigg\}^{\frac{1}{2}}  \\
&  \lesssim \bigg\{ \sum_{T \in \T_{\ba}}S_{{\dK}}(\hat{u}_T, \hat{u}_T)\bigg\}^{\frac{1}{2}}
\bigg\{\sum_{F \in \cF_a}\frac{h_F}{(k+2)^2}\|\curlrw{\mKM(\b{\ut{\psi}})}\b{n}_F\|^2_F\bigg\}^{\frac{1}{2}}.
\end{align*}
Using the trace inverse estimate \eqref{discrete} on $T\in\T_{\ba}$ so that $F\in\FK$, the stability of the interpolation operator $\mKM$ (see~\eqref{eq: Modified KM approximation}), and the stability of the local Helmholtz decomposition (see~\eqref{bound:Helmholtz}), we obtain
\begin{align}
\frac{h_F}{(k+2)^2}\|\curlrw{ \mKM(\b{\ut{\psi}})}\b{n}_F\|_F^2 &\lesssim \|
\curlrw{ \mKM(\b{\ut{\psi}})}\b{n}_F \|_{T} \nonumber \\
& \lesssim |\mKM(\b{\ut{\psi}})|_{\ut{\b{H}}^1(T)}
\lesssim |\ut{\b{\psi}}|_{\ut{\b{H}}^1(\omega_a)}
\lesssim \|\hes_{\mesh} \delta_{\ba}\|_{\oma}. \label{eq:bnd_trace_curlrw}
\end{align}
Combining the above two bounds yields
$$
|T_{21}| \lesssim \bigg\{ \sum_{T \in \T_{\ba}}S_{{\dK}}(\hat{u}_T, \hat{u}_T)\bigg\}^{\frac{1}{2}} \|\hes_{\mesh} \delta_{\ba}\|_{\oma}.
$$
To bound $T_{22}$, we observe that
\begin{align*}
T_{22} &=- \sum_{F \in \mathcal{F}_a}  \big((I - \Pi^k_F) \sjump{\nabla_{\mesh} u_{\mesh}}, (I - \Pi^k_{\tilde{T}})(\psi_{\ba}\curlrw{\mKM(\b{\ut{\psi}})}\b{n}_F)\big)_{F} \\
&\leq \sum_{F \in \mathcal{F}_a} \|(I - \Pi^k_F) \sjump{\nabla_{\mesh} u_{\mesh}} \|_{F}
\|(I - \Pi^k_{\tilde{T}}) (\psi_{\ba} \curlrw{\mKM(\b{\ut{\psi}})\b{n}_F}\|_{F},
\end{align*}
where $\tilde{T} \in \tilde{\mathcal{T}}_h$ denotes a simplex in the Alfeld split of which $F$ is a face. This simplex $\tilde{T}$ is a subset of a simplex $T\in\T_{\b{a}}$. Observing that $(\psi_{\ba}\curlrw {\mKM(\b{\ut{\psi}})}\b{n}_F)_{i} \in \mathbb{P}^{k+5}(\tilde{T})$ for all $i\in\{1{:}d\}$ and invoking the  discrete trace inequality~\eqref{eq: discrete trace sharp_estimate} with $p=k+5$ and $n=k$ (this is the crucial point where we gain a factor $(k+2)$ in the bound since $p-n=5$),  we infer that
\begin{align*}
\|(I - \Pi^k_{\tilde{T}}) (\psi_{\ba}\curlrw {\mKM(\b{\ut{\psi}})} \b{n}_F)\|_F
& \lesssim  \hbar_F^{-\frac12} \|(I - \Pi^k_{\tilde{T}}) (\psi_{\ba}\curlrw{\mKM(\b{\ut{\psi}})} \b{n}_F)\|_{\tilde{T}}\\
&\lesssim \hbar_F^{-\frac12} \|\curlrw {\mKM(\b{\ut{\psi}})}\|_{\tilde{T}}\\
& \lesssim \hbar_F^{-\frac12} \|\curlrw {\mKM(\b{\ut{\psi}})}\|_{T},
\end{align*}
where we also used that $k+5\lesssim k+2$. Moreover,
invoking the $hp$-approximation properties of $\Pi^k_F$ on $F$ gives
$$
\|(I - \Pi^k_F) \sjump{\nabla_{\mesh} u_{\mesh}} \|_F \lesssim
\hbar_F \big\{ \|\sjump{\partial_{tt} u_{\mesh}}\|_F+\|\sjump{\partial_{nt} u_{\mesh}}\|_F\big\}.
$$
Hence, putting the above two bounds together leads to
$$
|T_{22}| \lesssim \bigg\{ \sum_{F \in \mathcal{F}_a} \hbar_F \big\{ \|\sjump{\partial_{tt} u_{\mesh}}\|_F^2+\|\sjump{\partial_{nt} u_{\mesh}}\|_F^2\big\} \bigg\}^{\frac12} \|\hes_{\mesh} \delta_{\ba}\|_{\oma},
$$
where we invoked again~\eqref{eq:bnd_trace_curlrw}.
Finally, to bound $T_{23}$, we obtain using similar arguments that
\begin{align*}
|T_{23}| \lesssim {}& \sum_{F \in \Fa} \big\{ \|\Pi^k_F \sjump{\partial_n u_{\mesh} }\|_F
+ \|\Pi^k_F \sjump{\partial_t u_{\mesh} }\|_F\big\} \|\curlrw {\mKM(\b{\ut{\psi}})}\|_F.
\end{align*}
Reasoning as above and invoking again the bound~\eqref{eq:bnd_trace_curlrw} leads to
$$
|T_{23}| \lesssim \bigg\{ \sum_{T \in \T_{\ba}}S_{{\dK}}(\hat{u}_T, \hat{u}_T)\bigg\}^{\frac{1}{2}} \|\hes_{\mesh} \delta_{\ba}\|_{\oma}.
$$
Collecting the bounds on $T_{21}$, $T_{22}$, and $T_{23}$ gives
$$
|T_2| \lesssim \bigg\{ \sum_{T \in \T_{\ba}}S_{{\dK}}(\hat{u}_T, \hat{u}_T)
+ \sum_{F \in \mathcal{F}_a}  \hbar_F \big\{ \|\sjump{\partial_{tt} u_{\mesh}}\|_F^{2}+\|\sjump{\partial_{nt} u_{\mesh}}\|_F^{2}\big\} \bigg\}^{\frac{1}{2}} \|\hes_{\mesh} \delta_{\ba}\|_{\oma}.
$$
Finally, combining this bound on $T_2$ with the bound~\eqref{T1:BOUND} on $T_1$ completes the proof.

\begin{remark}[Application to IPDG]
To make the perspective of applying the above $H^2$-reconstruction
to IPDG methods, let us briefly discuss the minor modifications to the above proof. Since the IPDG method does not involve face unknowns, the term $T_1$ is bounded using \eqref{eq: IPDG T1}. Likewise, instead of splitting the term $T_2$ as in \eqref{eq: IPDG T2}, a discrete trace inequality can be applied directly. Consequently, for all $u_{\mesh}\in \mathbb{P}^{k+2}(\mesh)$, we obtain
\begin{align*}
\|\hes_{\mesh} (u_c - u_{\mesh})\|^2_{\Omega} &\lesssim \sum_{T \in \mesh} \big\{ \eta_{T, {\rm tan}}^2 + (k+2)\hbar_T^{-1}
 \|\sjump{\partial_{n} u_{\mesh}}\|_{\dK}^2
+
(k+2)^{3} \hbar_T^{-3}\|\sjump{u_{\mesh}}\|_{\dK}^2  \big\},
\end{align*}
and observe that the two rightmost terms correspond to the expected stabilization employed in the IPDG method.
\end{remark}

\bibliographystyle{siam}
\bibliography{bibliography.bib}
\end{document}